\documentclass[a4paper]{amsart}

\usepackage{amsfonts, amssymb, amsmath, amsthm}                               
\usepackage[font=scriptsize,labelfont=bf,margin=3cm]{caption}
\usepackage{epstopdf}
\usepackage{pdfpages}
\usepackage{color}
\usepackage{slashed}
\usepackage{comment}
\usepackage{enumerate}

\newtheorem{theorem}{Theorem}[section]      	      	                        
\newtheorem{corollary}[theorem]{Corollary}     	      	      	      	      
\newtheorem{lemma}[theorem]{Lemma}     	       	      	      	      	      
\newtheorem{proposition}[theorem]{Proposition} 	      	      	      	      
\newtheorem{definition}[theorem]{Definition}   	      	                      
\newtheorem{assumption}[theorem]{Assumption}   	      	                      
\newtheorem*{problem}{Problem}   	      	                                    
\theoremstyle{remark}
\newtheorem{remark}[theorem]{Remark}

\numberwithin{equation}{section}                                              
\numberwithin{theorem}{section}                                               
\numberwithin{figure}{section}                                                

\newcommand{\mf}[1]{\mathfrak{#1}}                                            
\newcommand{\mc}[1]{\mathcal{#1}}                                             

\newcommand{\N}{\mathbb{N}}                                                   
\newcommand{\R}{\mathbb{R}}                                                   

\newcommand{\ka}{\kappa}

\newcommand{\si}{\sigma}

\newcommand{\Ga}{\Gamma}

\newcommand{\Om}{\Omega}

\renewcommand\leq\leqslant
\renewcommand\geq\geqslant

\renewcommand\pi\Pi

\allowdisplaybreaks

\begin{document} 

\title[Approximate boundary controllability for parabolic equations]{Approximate boundary controllability\\ for parabolic equations with inverse\\ square infinite potential wells}

\author{Arick Shao}
\address{ School of Mathematical Sciences, Queen Mary University of London, London E1 4NS, United Kingdom}
\email{a.shao@qmul.ac.uk}
	
\author{Bruno Vergara}
\address{Department of Mathematics, Brown University, Kassar House, 151 Thayer St., Providence, RI 02912, USA}
\email{bruno\textunderscore vergara\textunderscore biggio@brown.edu}

\begin{abstract}
We consider heat operators on a bounded domain $\Omega \subseteq \R^n$, with a critically singular potential diverging as the inverse square of the distance to $\partial \Omega$.
Although null boundary controllability for such operators was recently proved in all dimensions in \cite{new}, it crucially assumed (i) $\Omega$ was convex, (ii) the control must be prescribed along all of $\partial \Omega$, and (iii) the strength of the singular potential must be restricted to a particular subrange.
In this article, we prove instead a definitive approximate boundary control result for these operators, in that we (i) do not assume convexity of $\Omega$, (ii) allow for the control to be localized near any $x_0 \in \partial \Omega$, and (iii) treat the full range of strength parameters for the singular potential.
Moreover, we lower the regularity required for $\partial \Omega$ and the lower-order coefficients.
The key novelty is a local Carleman estimate near $x_0$, with a carefully chosen weight that takes into account both the appropriate boundary conditions and the local geometry of $\partial \Omega$.
\vspace{0.1cm}

\smallskip

\noindent \textbf{Keywords.} Approximate control, parabolic equations, singular potentials, Carleman estimates, unique continuation.
\end{abstract}

\maketitle

\section{Introduction} \label{S.intro}

Throughout this article, we will consider the following setting:

\begin{assumption} \label{ass.domain}
Let $\Omega \subseteq \R^n$ be bounded, open, and connected, with $C^2$-boundary $\Gamma := \partial \Omega$.
Moreover, let $d_\Gamma: \Omega \rightarrow \R^+$ denote the distance to $\Gamma$.
\end{assumption}

Let us consider, on $\Omega$, heat operators with a potential that diverges as the inverse square of the distance to $\Gamma$.
More precisely, we consider the equation
\begin{equation}
\label{eq.0} -\partial_t v + \Big( \Delta + \frac{ \sigma }{ d_\Gamma^2 } \Big) v + Y \cdot \nabla v + W \, v = 0
\end{equation}
on $( 0, T ) \times \Omega$, for any $T > 0$.
Here, $\sigma \in \R$ is a parameter measuring the strength of the singular potential, while $Y$ and $W$ represent general first and zero-order coefficients that are less singular at $\Gamma$; see Definition \ref{def.lower_order} below.

Note that since the potential $\sigma d_\Gamma^{-2}$ scales as the Laplacian near $\Gamma$, one cannot simply treat \eqref{eq.0} as a perturbation of the standard heat equation.
Indeed, solutions to \eqref{eq.0} exhibit radically different behavior near $\Gamma$.
Also, the inclusion of $Y$, $W$ in \eqref{eq.0} is important in our context, as $d_\Ga$ can fail to be differentiable away from $\Ga$, hence lower-order corrections are needed for our heat operator to be regular.

While null boundary controllability for certain one-dimensional analogues of \eqref{eq.0} have been known for some time (see \cite{Biccari} and references therein), analogous results in higher dimensions were established only recently by the authors and Enciso in \cite{new}.
However, \cite{new} crucially assumed $\Gamma$ is convex and required the control to be set over all of $\Gamma$, but it was unclear if these conditions could be removed.
Also, the result of \cite{new} only held under an additional restriction $\sigma < 0$.

Our objective here is to establish instead an \emph{approximate} controllability result for \eqref{eq.0} that does not depend on any geometric assumptions on $\Gamma$, that allows for the control to be localized to arbitrarily small sectors of $\Gamma$, and that is applicable to all $\sigma$ for which the boundary control problem is well-defined.
In other words, given any time $T > 0$, any initial and final states $v_0$ and $v_T$, we find localized Dirichlet data $v_d$ such that the corresponding solution to \eqref{eq.0}, with the above initial data $v_0$ and boundary data $v_d$, becomes arbitrarily close to $v_T$ at time $T$.

While the existing null controllability result of \cite{new} hinges on global Carleman and observability estimates from the boundary for the adjoint of \eqref{eq.0}, approximate controllability only requires to establish a weaker unique continuation property from the boundary.
The novel contribution in this paper is a local Carleman estimate near the boundary that yields the requisite unique continuation.
Although such a local estimate bypasses key difficulties encountered in deriving the global estimates of \cite{new}, here we can significantly weaken our assumptions as mentioned above.

In addition, compared to \cite{new}, we impose weaker regularity assumptions for both the domain boundary $\Gamma$ and the lower-order coefficients $Y$, $W$.
In particular, here we only require $\Gamma$ to be $C^2$ (as opposed to $C^4$ in \cite{new}), and we require less differentiability for $Y$ and $W$ near $\Gamma$ (see Definition \ref{def.lower_order} below).

\subsection{Boundary asymptotics}

From here on, we assume the following for $\sigma$:

\begin{assumption} \label{ass.sigma}
Throughout the paper, we will assume
\begin{equation}
\label{eq.sigma} -\tfrac{3}{4} < \sigma < \tfrac{1}{4} \text{.}
\end{equation}
For convenience, we also define $\kappa := \kappa (\si) \in \R$ be the unique value satisfying
\begin{equation}
\label{eq.kappa} \sigma := \ka ( 1 - \ka ) \text{,} \qquad - \tfrac{1}{2} < \kappa < \tfrac{1}{2} \text{.}
\end{equation}
\end{assumption}

One consequence of the potential $\sigma d_\Gamma^{-2}$ is that it drastically alters both the well-posedness theory and the boundary asymptotics of solutions $v$ to \eqref{eq.0}, compared to the usual heat equation.
In particular, both the Dirichlet and Neumann branches of $v$ behave like specific powers of $d_\Gamma$ near $\Gamma$, with the exponent depending on $\sigma$.
Roughly, the expectation from ODE heuristics is that solutions will behave as
\begin{equation}
\label{eq.branches} v \simeq v_D \, d_\Gamma^\kappa + v_N \, d_\Gamma^{ 1 - \kappa }
\end{equation}
near $\Gamma$.
To capture this formally, we define the following boundary traces:

\begin{definition} \label{notations}
Supposing the setting of Assumptions \ref{ass.domain} and \ref{ass.sigma}:
\begin{itemize}
\item We define the associated Dirichlet and Neumann trace operators:
\begin{equation}
\label{eq.boundary_trace} \mc{D}_\sigma \phi := d_\Gamma^{ -\kappa } \phi |_{ d_\Gamma \searrow 0 } \text{,} \qquad \mc{N}_\sigma \phi := d_\Gamma^{ 2 \kappa } \nabla d_\Gamma \cdot \nabla ( d_\Gamma^{ -\kappa } \phi ) |_{ d_\Gamma \searrow 0 } \text{.}
\end{equation}

\item For convenience, we also set the following shorthand:
\begin{equation}
\label{eq.singular_op} \Delta_\sigma := \Delta + \sigma d_\Gamma^{-2} \text{.}
\end{equation}
\end{itemize}
\end{definition}

In particular, the traces \eqref{eq.boundary_trace} make precise the coefficients $v_D$ and $v_N$ (respectively) in \eqref{eq.branches}.
These traces play essential roles in the well-posedness theory of \eqref{eq.0} and its adjoint, and they serve to vindicate the boundary asymptotics suggested in \eqref{eq.branches}; see Section \ref{S.wp} below.
We note that similar boundary traces have also been constructed for analogously singular wave equations; see \cite{JEMS, Warnick}.

\begin{remark}
The restriction \eqref{eq.sigma} naturally arises from the well-posedness theory of \eqref{eq.0}.
First, we note that \eqref{eq.0} is expected to be ill-posed for $\sigma > \frac{1}{4}$ (see \cite{Baras, Biccari, VZuazua}), while boundary controllability is known to fail when $\sigma \nearrow \frac{1}{4}$ \cite{Biccari}.
In addition, note that when $\sigma \leq -\frac{3}{4}$ ($\kappa \leq -\frac{1}{2}$), the Dirichlet branch of solutions---that corresponding to $v_D \, d_\Gamma^\kappa$ in \eqref{eq.branches}---no longer lies in $L^2 ( \Omega )$.
\end{remark}

\subsection{Main results}

Next, we provide the precise assumptions that we impose on the lower-order coefficients $Y$ and $W$ in \eqref{eq.0}:

\begin{definition} \label{def.lower_order}
We let $\mc{Z}_0$ denote the collection of all pairs $( Y, W )$, where:
\begin{itemize}
\item $Y: \Omega \rightarrow \R^n$ is a $C^1$-vector field.

\item $W: \Omega \rightarrow \R$, and $d_\Gamma W \in L^\infty ( \Omega )$.
\end{itemize}
\end{definition}

\begin{remark}
Note Definition \ref{def.lower_order} is strictly weaker than the corresponding assumptions in \cite{new}, in that we require less differentiability for both $Y$ and $W$ near $\Gamma$.
\end{remark}

The main result of this paper is the following approximate boundary controllability property for the critically singular heat equation \eqref{eq.0}:

\begin{theorem} \label{T.approx0}
Suppose Assumptions \ref{ass.domain} and \ref{ass.sigma} hold.
Also, let $( Y, W ) \in \mc{Z}_0$, and fix any open $\omega \subseteq \Gamma$.
Then, given any $T > 0$, any $v_0, v_T \in H^{-1} (\Omega)$, and any $\epsilon > 0$, there exists $v_d \in L^2 ( (0, T) \times \Gamma )$, supported within $( 0, T ) \times \omega$, such that the solution $v$ to \eqref{eq.0}, with initial data $v(0)=v_0$ and Dirichlet trace $\mc{D}_\sigma v = v_d$, satisfies
\begin{equation}
\label{eq.app_ctrl} \| v(T) - v_T \|_{ H^{-1} (\Om) } \leq \epsilon \text{.}
\end{equation}
\end{theorem}

The key step in proving Theorem \ref{T.approx0} is a corresponding unique continuation property for the adjoint of \eqref{eq.0}---the backward heat equation
\begin{equation}
\label{eq.1} \partial_t u + \Big( \Delta + \frac{ \sigma }{ d_\Gamma^2 } \Big) u + X \cdot \nabla u + V \, u = 0 \text{.}
\end{equation}
Roughly, we show that \emph{if a ($H^1$-)solution $u$ to \eqref{eq.1} satisfies}
\begin{equation}
\label{eq.ucp_hyp} \mc{D}_\sigma u |_{ ( 0, T ) \times \omega } = \mc{N}_\sigma u |_{ ( 0, T ) \times \omega } = 0 \text{,}
\end{equation}
\emph{for any open $\omega \subseteq \Gamma$, then $u \equiv 0$ everywhere on $( 0, T ) \times \Omega$}.
The precise statement of this property is provided later in Theorem \ref{T.UC}.

In particular, once this unique continuation property is established, approximate controllability follows by adapting standard HUM arguments (see, e.g., \cite{Lions, MZ}) to \eqref{eq.0} and \eqref{eq.1} in the appropriate well-posed settings; see Section \ref{S.proof} for details.

\subsection{Well-posedness}

A crucial building block for Theorem \ref{T.approx0} and the HUM is a pair of dual well-posedness theories for \eqref{eq.0} and \eqref{eq.1}, in the $H^{-1}$ and $H^1$-levels, respectively.
While well-posedness for linear heat equations, at various regularity levels, is by now classical, the presence of the singular potentials and the ensuing modified boundary asymptotics complicate this process.

The well-posedness of \eqref{eq.0}, at the $L^2$ and $H^1$-levels, and with \emph{vanishing Dirichlet data}, was briefly described in \cite{VZuazua} using standard semigroup methods.
Moreover, these results held in a larger range $\sigma < \frac{1}{4}$, but they were only stated in the specific case $( Y, W ) \equiv ( 0, 0 )$.
(Note that \cite{VZuazua} treated the interior control problem, where one could avoid dealing with the modified boundary traces.)

The requisite well-posedness theory for \emph{nontrivial Dirichlet data}, and for a general class of $( Y, W )$, was summarized in \cite{new}.
In particular, \cite{new} constructed, from the $H^1$-theory mentioned above (lightly modified to account for nontrivial $( Y, W )$), a dual theory of transposition solutions at the $H^{-1}$-level, with prescribed inhomogeneous Dirichlet data.
A key point here is to show that the Neumann trace operator $\mc{N}_\sigma$ is well-defined in the $H^1$-theory, under the additional restriction $\sigma > -\frac{3}{4}$.
(As mentioned before, the condition $\sigma > -\frac{3}{4}$ is natural in this setting, since the Dirichlet branch no longer lies in $L^2$ when this is violated.)

Since we are treating the full range $-\frac{3}{4} < \sigma < \frac{1}{4}$ in this article (as opposed to only $-\frac{3}{4} < \sigma < 0$ in \cite{new}), and since we also weaken the regularity assumptions for $\Gamma$ and $( Y, W )$, here we elect, in Section \ref{S.wp}, to provide a more detailed treatment of the well-posedness theory for completeness.
In particular, we give a careful accounting of the existence and boundedness of the modified boundary traces \eqref{eq.boundary_trace}, and we clarify the spaces employed in various parts of the development.

Modifying the well-posedness theory to treat the weaker ($C^2$-)regularity for $\Gamma$ is straightforward, as this is simply a matter of noting that one never takes more than two derivatives of $d_\Gamma$ in the analysis.
The same is also mostly true for treating the weaker regularity of $( Y, Z )$; however, a notable exception is in the $H^{-1}$-theory for \eqref{eq.0} (the controllability side of the HUM), for which the argument requires a new technical ingredient---a notion of solution for \eqref{eq.0} in conjunction with a highly singular forcing term; see Proposition \ref{wp_weak} and its proof for details.

\begin{remark}
A well-posedness theory for wave equations with an analogously singular potential, with prescribed inhomogeneous boundary data, was established by Warnick \cite{Warnick} using Galerkin methods.
A similar approach could also be used here in order to treat lower-order coefficients $( Y, W )$ that are time-dependent.
\end{remark}

\subsection{Ideas of the proof}

The unique continuation property is proved via a new local Carleman estimate for \eqref{eq.1}---stated in Theorem \ref{T.Carleman}---that is supported near a point $x_0 \in \Gamma$.
The estimate itself is similar in structure to that of \cite{new}, in that it captures the Neumann data on $( 0, T ) \times \omega$; the reader is referred to discussions in \cite{new} for the basic ideas of the proof.
Here, we instead focus on the novel features in the present Carleman estimate that were not found in \cite{new}.

The first novelty is that the local Carleman estimate no longer requires convexity of $\Gamma$.
The key technical idea arises from a modifier to the zero-order term in the multiplier---the parameter $z > 0$ both here and in \cite{new}, which must be large enough to overcome any concavity in $\Gamma$.
In \cite{new}, the size of $z$ was further constrained by quantities in the interior of $\Omega$ that must be absorbed.
This constraint is not present here, since our estimate is localized near $\Gamma$, hence $z$ can be chosen as large as needed.

The second new feature is that the control is localized to a single $x_0 \in \Gamma$, which stems from a modification to the Carleman weight exponent.
In \cite{new}, the exponent
\[
( 1 + 2 \kappa )^{-1} d_\Gamma^{1+2\kappa}
\]
vanishes precisely on $\Gamma$; here, we adopt a modified exponent,
\begin{equation}
\label{eq.exponent} ( 1 + 2 \kappa )^{-1} d_\Gamma^{1+2\kappa} + | w |^2 \text{,}
\end{equation}
that vanishes only at $x_0$, with $w := ( w^1, \dots, w^{n-1} )$ being a completion of $d_\Gamma$ into a local coordinate system.
Most crucially, the $\nabla w^i$'s are constructed to be orthogonal to $\nabla d_\Gamma$, so that the interactions between $w$ and the singular potential $\sigma d_\Gamma^{-2}$ do not produce dangerous terms that cannot be treated.

The above points suffice to treat the range $-\frac{3}{4} < \sigma < 0$, however when $\sigma > 0$, one faces the same fundamental difficulties in establishing a Carleman estimate as in \cite{new}.
Nonetheless, for approximate control, we can sidestep this issue entirely, as we only require a unique continuation property for \eqref{eq.1}, rather than observability.
The crucial observation when $\sigma > 0$ is that if \eqref{eq.ucp_hyp} holds, then the Neumann trace $\mc{N}_\sigma u$ must vanish like a sufficiently positive power of $d_\Gamma$; see Proposition \ref{neumann_limits}.
In other words, $u$ vanishes at $\Gamma$ like (in fact, better than) solutions of \eqref{eq.1} with $\sigma < 0$.
This then allows us to \emph{apply instead the above-mentioned Carleman estimates with weight corresponding to a negative $\sigma$}, for which we are now able to extract the positivity on the bulk terms necessary for the unique continuation property.


\begin{remark}
It is expected that null controllability should still hold without assuming convexity for $\Omega$, however this will be pursued in a future paper.
On the other hand, it is not yet clear whether one can localize null controls near a single $x_0 \in \Gamma$, as the Carleman estimates in \cite{new} crucially relied on weights constructed from $d_\Gamma$ itself.
Similarly, null controllability in the range $0 < \sigma < \frac{1}{4}$ remains an open question; however, intuitions from \cite{Gueye} seem to indicate that a proof via HUM and observability may only hold through a sharper well-posedness theory involving fractional Sobolev spaces.
\end{remark}

Finally, similar to the well-posedness theory, to treat the case of $\Gamma$ being only $C^2$, we modify the derivation of our Carleman estimate (compared to \cite{new}) so that we only take two derivatives of $d_\Gamma$.
The key observation is that only the most singular terms need to be treated via integrations by parts that lead to differentiating $d_\Gamma$, and those terms do not contain derivatives of $d_\Gamma$ to begin with.

\begin{remark}
The above regularity improvements can also be applied to \cite{new}, so that the null controllability results in \cite{new} also hold for $\Gamma \in C^2$ and $( Y, W ) \in \mc{Z}_0$.
\end{remark}

\subsection{Previous results}

Parabolic equations with inverse square potentials have attracted much attention over the last decades, though most existing results treat potentials $\sigma |x|^{-2}$ diverging at a single point; for some early results, see \cite{Baras, Ireneo}.
For conciseness, we limit our discussions to controllability of such equations.

For $n = 1$, there is a vast source of literature addressing the operator
\begin{equation}
\label{eq.heat_1d} -\partial_t + \partial_x^2 + \sigma \, x^{-2}
\end{equation}
on $\Omega := ( 0, 1 )$, see, e.g., \cite{Biccari, CMV, CMV4, CTY, Gueye, MV}.
(Note, however, that \eqref{eq.heat_1d} is not quite an analogue of \eqref{eq.0}, as the potential is regular at $1 \in \Gamma$.)
In particular, Biccari \cite{Biccari} proved, via the moment method, boundary null controllability from $x = 0$ for $-\frac{3}{4} < \sigma < \frac{1}{4}$.
In addition, \cite{Biccari} highlighted the difficulty of extending the result to higher dimensions, as well as the desirability of robust methods---e.g.\ via Carleman estimates---that could extend to nonlinear equations.

In higher dimensions, with general $\Omega \subseteq \R^n$, \cite{Cazacu, Ervedoza, VZuazua} established interior null controllability results for the singular heat operator
\begin{equation}
\label{eq.0_pt} - \partial_t + \Delta + \sigma \, | x - x_0 |^{-2} \text{,}
\end{equation}
with either $x_0 \in \Omega$ or $x_0 \in \Gamma$.
However, the setting \eqref{eq.0}, in which the potential becomes singular on all of $\Gamma$, has long been known to be especially difficult.
In \cite{BZuazua}, Biccari and Zuazua proved interior null controllability for $-\partial_t + \Delta_\sigma$ using Carleman estimates, but the same techniques could not be used for boundary control as the Carleman estimates fail to capture the full $H^1$-energy and the boundary data \eqref{eq.boundary_trace}.

Very recently, in \cite{new}, the present authors, along with A.\ Enciso, established the first boundary null controllability result for \eqref{eq.0} in all spatial dimensions, under the additional assumption that $\Gamma$ is convex.
Since the proof employs Carleman estimates, the result is robust in that one can consider general lower-order coefficients $Y$, $W$.
Additionally, in contrast to the Carleman estimates in~\cite{BZuazua}, the estimates of \cite{new} capture the appropriate notion of the Neumann data at the boundary and the natural $H^1$-energy, both of which are crucial for boundary null control.

Lastly, there is also an extensive body of approximate controllability results for linear and nonlinear parabolic equations---see, for example, \cite{Fabre, FPZ, FCZ1, FCZ2, Lions, MZ}, though this list is nowhere near complete.

\subsection{Outline of the paper}

In Section \ref{S.wp}, we discuss the relevant well-posedness theories for \eqref{eq.0} and its adjoint \eqref{eq.1}.
The heart of the analysis lies in Section \ref{S.carleman}, which presents the novel local Carleman estimate for \eqref{eq.1}.
The key unique continuation property for \eqref{eq.1} is then proved in Section \ref{S.uc}, while our main approximate control result, Theorem \ref{T.approx0}, is proved in Section \ref{S.proof}.

\section{Well-Posedness} \label{S.wp}

In this section, we give a self-contained presentation of the well-posedness theories needed for \eqref{eq.0} and \eqref{eq.1}.
As is standard in HUM proofs, we will need dual theories for the controllability and observability settings.

For the \emph{observability} side, we will consider the following two problems:

\begin{problem}[OI]
Given final data $u_T$ on $\Omega$, and forcing term $F$ on $( 0, T ) \times \Omega$, solve the following final-boundary value problem for $u$,
\begin{align}
\label{heat_ex} ( \partial_t + \Delta_\sigma + X \cdot \nabla + V ) u = F &\quad \text{on $( 0, T ) \times \Omega$,} \\
\notag u ( T ) = u_T &\quad \text{on $\Omega$,} \\
\notag \mc{D}_\sigma u = 0 &\quad \text{on $( 0, T ) \times \Gamma$,}
\end{align}
where the lower-order coefficients satisfy $( X, V ) \in \mc{Z}_0$.
\end{problem}

\begin{problem}[O]
Given final data $u_T$ on $\Omega$, solve the following for $u$,
\begin{align}
\label{heat_obs} ( \partial_t + \Delta_\sigma + X \cdot \nabla + V ) u = 0 &\quad \text{on $( 0, T ) \times \Omega$,} \\
\notag u ( T ) = u_T &\quad \text{on $\Omega$,} \\
\notag \mc{D}_\sigma u = 0 &\quad \text{on $( 0, T ) \times \Gamma$,}
\end{align}
where the lower-order coefficients satisfy $( X, V ) \in \mc{Z}_0$.
\end{problem}

The corresponding problem on the \emph{controllability} side is the following:

\begin{problem}[C]
Given initial data $v_0$ on $\Omega$, as well as Dirichlet boundary data $v_d$ on $( 0, T ) \times \Ga$, solve the initial-boundary value problem for $v$,
\begin{align}
\label{heat_ctl} -\partial_t v + \Delta_\sigma v + Y \cdot \nabla v + W v = 0 &\quad \text{on $( 0, T ) \times \Omega$,} \\
\notag v ( 0 ) = v_0 &\quad \text{on $\Omega$,} \\
\notag \mc{D}_\sigma v = v_d &\quad \text{on $( 0, T ) \times \Gamma$} \text{,}
\end{align}
where the lower-order coefficients satisfy $( Y, W ) \in \mc{Z}_0$.
\end{problem}

\begin{remark}
We employ the labels (O), (OI), (C) to maintain consistency with \cite{new}.
Note that Problem (O) is simply Problem (OI) in the special case $F \equiv 0$.
\end{remark}

\subsection{Preliminaries}

As usual, we define $H^1_0 ( \Omega )$ to be the closure, in the $H^1 ( \Omega )$-norm, of the space $C^\infty_0 ( \Omega )$ of smooth functions on $\Omega$ with compact support, and we define $H^{-1} ( \Omega )$ be the Hilbert space dual of $H^1_0 ( \Omega )$.

\begin{definition} \label{op}
For convenience, we define the following quantities,
\begin{align}
\label{AB_sigma} A_\sigma := \Delta_\sigma + X \cdot \nabla + V \text{,} \qquad B_\sigma := \Delta_\sigma + Y \cdot \nabla + W \text{,}
\end{align}
which we can also view as unbounded operators on appropriate spaces:
\begin{align}
\label{AB_op} A_\sigma: \mf{D} ( A_\sigma ) := \{ \phi \in H^1_0 ( \Omega ) \mid A_\sigma \phi \in L^2 ( \Omega ) \} \rightarrow L^2 ( \Omega ) \text{,} \\
\notag B_\sigma: \mf{D} ( B_\sigma ) := \{ \phi \in H^1_0 ( \Omega ) \mid B_\sigma \phi \in L^2 ( \Omega ) \} \rightarrow L^2 ( \Omega ) \text{.}
\end{align}
\end{definition}

\begin{remark}
For approximation purposes, it will also be useful to consider
\begin{align}
\label{AB_2} \mf{D} ( A_\sigma^2 ) &:= \{ \phi \in \mf{D} ( A_\sigma ) \mid A_\sigma \phi \in \mf{D} ( A_\sigma ) \} \text{,}
\end{align}
and similarly for $B_\sigma$.
Note in particular that $\mf{D} ( A_\sigma )$ and $\mf{D} ( A_\sigma^2 )$ are dense in $L^2 ( \Omega )$, $H^1_0 ( \Omega )$, and $H^{-1} ( \Omega )$, since all these domains contain $C^\infty_0 ( \Omega )$ by definition.
\end{remark}

Recall that $d_\Gamma$ could fail to be differentiable away from $\Gamma$.
Thus, for our analysis, we will need to work with a sufficiently smooth correction to $d_\Gamma$:

\begin{definition} \label{bdf}
We fix a \emph{boundary defining function} $y \in C^2 ( \Omega )$ satisfying:
\begin{itemize}
\item $y > 0$ everywhere on $\Omega$.

\item $y$ and $d_\Gamma$ coincide on a neighborhood of $\Gamma$.
\end{itemize}
Furthermore, for any $q \in \R$, we define the following shorthands:
\begin{equation}
\label{twisted_deriv} D_y := \nabla y \cdot \nabla \text{,} \qquad \nabla_q := y^q \nabla y^{-q} \text{.}
\end{equation}
\end{definition}

\begin{remark}
It is often more convenient to rewrite our operators in the form
\begin{equation}
\label{mod_op} A_\sigma = \nabla_{ -\kappa } \cdot \nabla_\kappa + X \cdot \nabla + V_y \text{,} \qquad B_\sigma = \nabla_{ -\kappa } \cdot \nabla_\kappa + Y \cdot \nabla + W_y \text{,}
\end{equation}
where the modified potentials $V_y$ and $W_y$ are given by
\begin{equation}
\label{VW_y} V_y - V = W_y - W = \kappa y^{-1} \Delta y + \sigma ( d_\Gamma^{-2} - y^{-2} | \nabla y |^2 ) \text{.}
\end{equation}
Note that since $y \in C^2 ( \Omega )$ and $y = d_\Gamma$ near $\Gamma$, then $( X, V_y ), ( Y, W_y ) \in \mc{Z}_0$.
\end{remark}

The subsequent construction will be helpful for dealing with boundary traces:

\begin{definition} \label{traces}
Let $0 < \delta \ll 1$ be small enough so that $y = d_\Gamma$ on $\{ y = \delta \}$.
\begin{itemize}
\item Given any $x \in \Gamma$, we let $x_\delta \in \{ y = \delta \}$ denote the point that is connected to $x$ along an integral curve of $\nabla y$ lying within $\{ y \leq \delta \}$.

\item For any $\phi \in H^1_{ \mathrm{loc} } ( \Omega )$, we define its trace on $\{ y = \delta \}$ by
\begin{equation}
\label{traces_space} \eta^\delta \phi \in L^2 ( \Gamma ) \text{,} \qquad \eta^\delta ( \phi ) (x) := \phi ( x_\delta ) \text{.}
\end{equation}

\item Given $\psi \in L^2 ( ( 0, T ); H^1_{ \mathrm{loc} } ( \Omega ) )$, we similarly define its trace by
\begin{equation}
\label{traces_time} \eta^\delta \psi \in L^2 ( ( 0, T ) \times \Gamma ) \text{,} \qquad \eta^\delta ( \psi ) ( t, x ) := \psi ( t, x_\delta ) \text{.}
\end{equation}
\end{itemize}
(The restrictions \eqref{traces_space} and \eqref{traces_time} are defined in the sense of Sobolev traces.)
\end{definition}

\begin{remark}
In particular, we can precisely define our boundary trace operators $\mc{D}_\sigma$ and $\mc{N}_\sigma$ as limits of $\eta^\delta$'s in Definition \ref{traces} as $\delta \searrow 0$.
\end{remark}

Next, we collect some key properties of $H^1_0 ( \Omega )$, in particular with regards to our singular potential.
We first recall the Hardy inequality that was proved in \cite{BM}:

\begin{proposition}[Hardy inequality \cite{BM}] \label{prop_hardy}
There exists $c \in \R$, depending on $\Omega$, with
\begin{equation}
\label{hardy_main} \tfrac{1}{4} \int_\Omega d_\Gamma^{-2} \phi^2 \leq \int_\Omega | \nabla \phi |^2 + c \int_\Omega \phi^2 \text{,} \qquad \phi \in H^1_0 ( \Omega ) \text{.}
\end{equation}
\end{proposition}

\begin{remark}
When $\Omega$ is convex, \cite{BM} also showed that $c < 0$ in \eqref{hardy_main}.
\end{remark}

\begin{corollary} \label{cor_hardy}
The following holds, with constants depending only on $\Omega$ and $\sigma$:
\begin{equation}
\label{hardy_norm_equiv} \| \phi \|_{ H^1 ( \Omega ) }^2 \simeq \| \nabla_\kappa \phi \|_{ L^2 ( \Omega ) }^2 + \| \phi \|_{ L^2 ( \Omega ) }^2 \text{,} \qquad \phi \in H^1_0 ( \Omega ) \text{.}
\end{equation}
\end{corollary}

\begin{proof}
Assume first $\phi \in C^\infty_0 ( \Omega )$.
By \eqref{mod_op}--\eqref{VW_y} and an integration by parts,
\begin{align*}
\int_\Omega | \nabla_\kappa \phi |^2 &= - \int_\Omega ( \phi \nabla_{ -\kappa } \cdot \nabla_\kappa \phi ) \\
&= - \int_\Omega \phi ( \Delta + \sigma d_\Gamma^{-2} ) \phi + \int_\Omega \phi [ \kappa y^{-1} \Delta y + \sigma ( d_\Gamma^{-2} - y^{-2} | \nabla y |^2 ) ]\phi \\
&\geq \int_\Omega | \nabla \phi |^2 - \sigma \int_\Omega d_\Gamma^{-2} \phi^2 - C \| y^{-1} \phi \|_{ L^2 ( \Omega ) } \| \phi \|_{ L^2 ( \Omega ) } \text{,}
\end{align*}
for some $C > 0$ depending on $\Omega$, $\sigma$.
Applying \eqref{hardy_main} to the above then yields
\begin{align*}
\| \nabla_\kappa \phi \|_{ L^2 ( \Omega ) }^2 &\geq [ 1 - 4 \max( \sigma, 0 ) - \delta ] \| \nabla \phi \|_{ L^2 ( \Omega ) }^2 - C \| \phi \|_{ L^2 ( \Omega ) }^2 \\
&\geq c \| \nabla \phi \|_{ L^2 ( \Omega ) }^2 - C \| \phi \|_{ L^2 ( \Omega ) }^2 \text{,}
\end{align*}
again for some constants $c, C > 0$ depending on $\Omega$, $\sigma$ and any $0 < \delta \ll 1$.
(Note that the last step follows by the assumption $4 \sigma < 1$ and by a choice of a sufficiently small $\delta$.)
In particular, the above yields half of \eqref{hardy_norm_equiv}.

The remaining part of \eqref{hardy_norm_equiv} also follows from \eqref{hardy_main} but is easier:
\begin{align*}
\| \nabla_\kappa \phi \|_{ L^2 ( \Omega ) } + \| \phi \|_{ L^2 ( \Omega ) } &\lesssim \| \nabla \phi \|_{ L^2 ( \Omega ) } + \| y^{-1} \phi \|_{ L^2 ( \Omega ) } \\
&\lesssim \| \phi \|_{ H^1 ( \Omega ) } \text{.}
\end{align*}
Finally, \eqref{hardy_norm_equiv} for $\phi \in H^1_0 ( \Omega )$ follows via a standard approximation argument.
\end{proof}

\begin{proposition} \label{dirichlet_vanish}
If $\phi \in H^1_0 ( \Omega )$, then the following hold:
\begin{itemize}
\item $\mc{D}_\sigma \phi$ is well-defined in $L^2 ( \Gamma )$, and $\mc{D}_\sigma \phi \equiv 0$---in other words,
\begin{equation}
\label{dirichlet_zero} \lim_{ \delta \searrow 0 } \| \eta^\delta ( y^{ -\kappa } \phi ) \|_{ L^2 ( \Gamma ) } = 0 \text{.}
\end{equation}

\item The following estimate holds, with the constant depending on $\Omega$:
\begin{equation}
\label{dirichlet_extra} \limsup_{ \delta \searrow 0 } \int_{ \{ y = \delta \} } y^{-1} \phi^2 \lesssim \| \phi \|_{ H^1 ( \Omega ) }^2 \text{.}
\end{equation}
\end{itemize}
\end{proposition}

\begin{proof}
Given $0 < \delta \ll 1$, we apply the divergence theorem and \eqref{hardy_main} to obtain
\begin{align*}
\int_{ \{ y = \delta \} } y^{-1} \phi^2 &= \int_{ \{ y > \delta \} } \nabla \cdot ( y^{-1} \nabla y \, \phi^2 ) \\
&\lesssim \| y^{-1} \phi \|_{ L^2 ( \Omega ) } ( \| \nabla y \|_{ L^2 ( \Omega ) } + \| y^{-1} \phi \|_{ L^2 ( \Omega ) } ) \\
&\lesssim \| \phi \|_{ H^1 ( \Omega ) }^2 \text{,}
\end{align*}
where we also recalled that $y = d_\Gamma$ near $\Gamma$; the above yields the bound \eqref{dirichlet_extra}.
Since $2 \kappa < 1$, then \eqref{dirichlet_extra} also implies both \eqref{dirichlet_zero} and $\mc{D}_\sigma \phi = 0$, since
\begin{align*}
\| \eta^\delta ( y^{ -\kappa } \phi ) \|_{ L^2 ( \Gamma ) }^2 &\lesssim \int_{ \{ y = \delta \} } y^{ 1 - 2 \kappa } \cdot y^{-1} \phi^2 \\
&\lesssim \delta^{ 1 - 2 \kappa } \| \phi \|_{ H^1 ( \Omega ) }^2 \text{.} \qedhere
\end{align*}
\end{proof}

\begin{remark}
Note in particular that $\sigma = \frac{1}{4}$ ($\kappa = \frac{1}{2}$) represents the threshold at which $H^1_0 ( \Omega )$ no longer adequately captures vanishing Dirichlet trace.
\end{remark}

Finally, we recall from \cite{new} a pointwise extension of the Hardy inequality \eqref{hardy_main} that will be essential to our upcoming Carleman estimate:

\begin{proposition} \label{prop_hardy_gen}
Let $\phi \in H^1_{ \mathrm{loc} } ( \Omega )$ and $q \in \R$.
Then, almost everywhere on $\Omega$,
\begin{align}
\label{hardy_gen} y^{2q} ( D_y \phi )^2 &\geq \nabla \cdot \big[ \tfrac{1}{2} ( 1 - 2q ) y^{ 2q - 1 } \nabla y | \nabla y |^2 \, \phi^2 \big] + \tfrac{1}{4} (1 - 2q )^2 y^{ 2q - 2 } | \nabla y |^4 \, \phi^2 \\
\notag &\qquad - \tfrac{1}{2} ( 1 - 2q ) y^{ 2q - 1 } [ \Delta y | \nabla y |^2 + 2 ( \nabla y \cdot \nabla^2 y \cdot \nabla y ) ] \, \phi^2 \text{.}
\end{align}
\end{proposition}

\begin{proof}
A direct computation yields, for any $b, q \in \R$, the inequality
\begin{align*}
0 &\leq ( y^q \, D_y \phi + b y^{ q - 1 } | \nabla y |^2 \, \phi )^2 \\
&= y^{ 2q } \, ( D_y \phi )^2 + b ( b - 2q + 1 ) y^{ 2q - 2 } | \nabla y |^4 \, \phi^2 - 2 b y^{ 2q - 1 } ( \nabla y \cdot \nabla^2 y \cdot \nabla y ) \, \phi^2 \\
&\qquad - b y^{ 2q - 1 } \Delta y | \nabla y |^2 \, \phi^2 + \nabla \cdot ( b y^{ 2q - 1 } \nabla y | \nabla y |^2 \, \phi^2 ) \text{.}
\end{align*}
Taking the optimal value $2b := 2q - 1$ in the above yields \eqref{hardy_gen}.
\end{proof}

\subsection{Elliptic properties}

The next step is to derive various elliptic properties for the operator $-A_\sigma$ (or equivalently, $-B_\sigma$) from Definition \ref{op}:

\begin{proposition} \label{prop_resolvent}
There exist $\gamma \geq 0$, $c > 0$ (depending on $\Omega$, $\sigma$, $X$, $V$) such that:
\begin{itemize}
\item The operator $\lambda I - A_\sigma: \mf{D} ( A_\sigma ) \rightarrow L^2 ( \Omega )$ is invertible for any $\lambda > \gamma$.

\item The following estimate holds for any $\lambda > \gamma$ and $f \in L^2 ( \Omega )$:
\begin{equation}
\label{resolvent} c \| \nabla ( \lambda I - A_\sigma )^{-1} f \|_{ L^2 ( \Omega ) } + ( \lambda - \gamma ) \| ( \lambda I - A_\sigma )^{-1} f \|_{ L^2 ( \Omega ) } \leq \| f \|_{ L^2 ( \Omega ) } \text{.}
\end{equation}
\end{itemize}
\end{proposition}

\begin{proof}
First, by \eqref{mod_op}--\eqref{VW_y}, we can associate to $-A_\sigma$ the bilinear form
\begin{equation}
\label{elliptic_bilinear} \mc{B}_\sigma ( \phi, \psi ) := \int_\Omega [ \nabla_\kappa \phi \cdot \nabla_\kappa \psi - ( X \cdot \nabla \phi ) \, \psi - V_y \, \phi \psi ] \text{,}
\end{equation}
defined for $\phi, \psi \in H^1_0 ( \Omega )$.
By Definition \ref{def.lower_order}, \eqref{hardy_main}, and \eqref{hardy_norm_equiv}, we have
\begin{align}
\label{elliptic_energy} \mc{B}_\sigma ( \phi, \phi ) &= \| \nabla_\kappa \phi \|_{ L^2 ( \Omega ) }^2 - C \| \phi \|_{ H^1 ( \Omega ) } \| \phi \|_{ L^2 ( \Omega ) } \\
\notag &\geq c \| \phi \|_{ H^1 ( \Omega ) }^2 - \gamma \| \phi \|_{ L^2 ( \Omega ) }^2
\end{align}
for any $\phi \in H^1_0 ( \Omega )$, where $C > 0$, $c > 0$, and $\gamma \geq 0$ are constants depending on $\Omega$, $\sigma$, $X$, $V$.
Consequently, given any $\lambda > \gamma$, the Lax-Milgram theorem and \eqref{elliptic_energy} yield for any $f \in L^2 ( \Omega )$ a unique $\phi_f \in H^1_0 ( \Omega )$ satisfying
\begin{equation}
\label{elliptic_weak} \lambda \int_\Omega \phi_f \psi + \mc{B}_\sigma ( \phi_f, \psi ) = \int_\Omega f \psi \text{,} \qquad \psi \in H^1_0 ( \Omega ) \text{.}
\end{equation}

Integrating \eqref{elliptic_weak} by parts and taking $\psi \in C^\infty_0 ( \Omega )$ (to remove boundary terms), we see that $f = ( \lambda I - A_\sigma ) \phi_f$.
It also follows that $\lambda I - A_\sigma$ is invertible, since
\[
A_\sigma \phi_f = - f + \lambda \phi_f \in L^2 ( \Omega ) \text{,} \qquad \phi_f \in \mf{D} ( A_\sigma ) \text{.}
\]
Setting $\psi := \phi_f$ in \eqref{elliptic_weak} and recalling \eqref{elliptic_energy}, we obtain
\[
c \| \phi_f \|_{ H^1 ( \Omega ) }^2 + ( \lambda - \gamma ) \| \phi_f \|_{ L^2 ( \Omega ) }^2 \leq \| f \|_{ L^2 ( \Omega ) } \| \phi_f \|_{ L^2 ( \Omega ) } \text{,}
\]
from which the desired bound \eqref{resolvent} follows.
\end{proof}

\begin{proposition} \label{elliptic}
The following holds for any $\phi \in \mf{D} ( A_\sigma )$:
\begin{itemize}
\item $\phi \in H^2_{ \mathrm{loc} } ( \Omega )$, and $\phi$ satisfies the following estimate:
\begin{equation}
\label{elliptic_H2} \| \nabla_{ -\kappa } \nabla_\kappa \phi \|_{ L^2 ( \Omega ) } + \| \nabla \phi \|_{ L^2 ( \Omega ) } \lesssim \| A_\sigma \phi \|_{ L^2 ( \Omega ) } + \| \phi \|_{ L^2 ( \Omega ) } \text{,}
\end{equation}

\item Furthermore, $\mc{N}_\sigma \phi$ is well-defined in $L^2 ( \Gamma )$---that is, $\eta^\delta [ y^{ 2 \kappa } D_y ( y^{ -\kappa } \phi ) ]$ has a limit in $L^2 ( \Gamma )$ as $\delta \searrow 0$---and the following estimate holds:
\begin{equation}
\label{elliptic_neumann} \| \mc{N}_\sigma \phi \|_{ L^2 ( \Gamma ) } \lesssim \| A_\sigma \phi \|_{ L^2 ( \Omega ) } + \| \phi \|_{ L^2 ( \Omega ) } \text{,}
\end{equation}
\end{itemize}
In both inequalities above, the constants depend only on $\Omega$, $\sigma$, $X$, $V$.
\end{proposition}

\begin{proof}
Letting $\lambda$ be as in Proposition \ref{resolvent}, then the estimate \eqref{resolvent} yields
\begin{align}
\label{elliptic_H1} \| \nabla \phi \|_{ L^2 ( \Omega ) } &\lesssim \| ( \lambda I - A_\sigma ) \phi \|_{ L^2 ( \Omega ) } \\
\notag &\lesssim \| A_\sigma \phi \|_{ L^2 ( \Omega ) } + \| \phi \|_{ L^2 ( \Omega ) } \text{.}
\end{align}
In addition, since the coefficients of $A_\sigma$ are bounded on any compact subset of $\Omega$, then standard interior elliptic regularity (see, e.g., \cite{Evans, Krylov} and references therein) implies $\phi \in H^2_{ \mathrm{loc} } ( \Omega )$.
Therefore, we need only to bound $\nabla_{ -\kappa } \nabla_\kappa \phi$ and $\mc{N}_\sigma \phi$ in \eqref{elliptic_H2} and \eqref{elliptic_neumann}, for $\phi$ supported near any $x_0 \in \Gamma$.

Next, the idea for controlling $\nabla_{ -\kappa } \nabla_\kappa \phi$ is once again similar to standard elliptic estimates.
By \eqref{mod_op} and an integration by parts, we have that
\begin{align*}
\lambda \int_\Omega [ \nabla_\kappa \phi \cdot \nabla_\kappa \psi - ( X \cdot \nabla \phi ) \, \psi - V_y \, \phi \psi ] = \int_\Omega A_\sigma \phi \cdot \psi
\end{align*}
for any $\psi \in C^\infty_0 ( \Omega )$, and hence for any $\psi \in H^1_0 ( \Omega )$ by approximation.
Now, let $\slashed{\nabla}$ and $\slashed{\Delta}$ denote the gradient and Laplacian on level sets of $y$, respectively, and let $\slashed{\nabla}_\star$ and $\slashed{\Delta}_\star$ denote corresponding difference quotients.
Letting $\psi := \slashed{\Delta}_\star \phi \in H^1_0 ( \Omega )$ in the above and recalling \eqref{hardy_main}--\eqref{hardy_norm_equiv}, we then obtain
\begin{align*}
\| \slashed{\nabla}_\star \nabla_\kappa \phi \|_{ L^2 ( \Omega ) }^2 &= \int_\Omega \nabla_\kappa \phi \cdot \slashed{\Delta}_\star \nabla_\kappa \phi \\
&\lesssim ( \| A_\sigma \phi \|_{ L^2 ( \Omega ) } + \| \phi \|_{ H^1 ( \Omega ) } ) \| \slashed{\nabla}_\star \nabla_\kappa \phi \|_{ L^2 ( \Omega ) } + \| \phi \|_{ H^1 ( \Omega ) }^2 \text{.}
\end{align*}
Combining the above with \eqref{elliptic_H1} yields the bound
\[
\| \slashed{\nabla} \nabla_\kappa \phi \|_{ L^2 ( \Omega ) }^2 \lesssim \| A_\sigma \phi \|_{ L^2 ( \Omega ) }^2 + \| \phi \|_{ L^2 ( \Omega ) }^2 \text{.}
\]
(Note we used standard properties of difference quotients in the above; see \cite{Evans}.)

The remaining second derivative $y^{ -\kappa } D_y [ y^{ 2 \kappa } D_y ( y^{ -\kappa } \phi ) ]$ can now be controlled---with the help of \eqref{hardy_norm_equiv} and the above---by $A_\sigma \phi$, $\slashed{\nabla} \nabla_\kappa \phi$, $X \cdot \nabla \phi$, and $V_y \phi$, yielding
\begin{align*}
\| \nabla_{ -\kappa } \nabla_\kappa \phi \|_{ L^2 ( \Omega ) } &\lesssim \| A_\sigma \phi \|_{ L^2 ( \Omega ) } + \| \slashed{\nabla} \nabla_\kappa \phi \|_{ L^2 ( \Omega ) } + \| \nabla \phi \|_{ L^2 ( \Omega ) } + \| y^{-1} \phi \|_{ L^2 ( \Omega ) } \\
&\lesssim \| A_\sigma \phi \|_{ L^2 ( \Omega ) } + \| \phi \|_{ H^1 ( \Omega ) } \text{.}
\end{align*}
The inequality \eqref{elliptic_H2} now follows from combining \eqref{elliptic_H1} and the above.

For $\mc{N}_\sigma \phi$, first observe that for any $0 < y_1 < y_0 \ll 1$, we have
\begin{align*}
&\| \eta^{ y_0 } y^{ 2 \kappa } D_y ( y^{ -\kappa } \phi ) - \eta^{ y_1 } y^{ 2 \kappa } D_y ( y^{ -\kappa } \phi ) \|^2_{ L^2 ( \Gamma ) } \\
&\quad = \int_\Gamma \bigg( \int_{ y_1 }^{ y_0 } \eta^s D_y [ y^{ 2 \kappa } D_y ( y^{ -\kappa } \phi ) ] \, ds \bigg)^2 \\
&\quad \lesssim \int_{ y_1 }^{ y_0 } s^{ 2 \kappa } ds \cdot ( \| \nabla_{ -\kappa } \nabla_\kappa \phi \|_{ L^2 ( \Omega ) }^2 + \| \phi \|_{ H^1 ( \Omega ) }^2 ) \\
&\quad \lesssim y_0^{ 1 + 2 \kappa } ( \| A_\sigma \phi \|_{ L^2 ( \Omega ) }^2 + \| \phi \|_{ L^2 ( \Omega ) }^2 ) \text{,}
\end{align*}
where we applied \eqref{elliptic_H2} in the last step, and we where noted $2 \kappa > -1$.
In particular, the right-hand side vanishes $y_0 \searrow 0$, which implies $\mc{N}_\sigma \phi$ is well-defined in $L^2 ( \Gamma )$.

Next, we fix $0 < y_0 \ll 1$ and a smooth cutoff $\chi: ( 0, \infty ) \rightarrow [ 0, 1 ]$ satisfying
\begin{equation}
\label{cutoff_y} \chi (s) = \begin{cases} 1 & 0 < s < y_0 \text{,} \\ 0 & s > 2 y_0 \text{.} \end{cases}
\end{equation}
Then, a similar estimate as before, using also the above $\chi$, yields
\begin{align*}
\int_\Gamma ( \mc{N}_\sigma \phi )^2 &= \int_\Gamma \bigg( \int_0^{ 2 y_0 } \eta^s D_y [ \chi (y) \cdot y^{ 2 \kappa } D_y ( y^{ -\kappa } \phi ) ] \, ds \bigg)^2 \\
&\lesssim \| \nabla_{ -\kappa } \nabla_\kappa \phi \|_{ L^2 ( \Omega ) }^2 + \| \nabla_\kappa \phi \|_{ L^2 ( \Omega ) }^2 \text{,}
\end{align*}
so that combining \eqref{hardy_norm_equiv}, \eqref{elliptic_H2}, and the above results in \eqref{elliptic_neumann}.
\end{proof}

\begin{corollary} \label{semigroup}
There exists some $\gamma \geq 0$ (depending on $\Omega$, $\sigma$, $X$, $V$) such that $-A_\sigma$ generates a $\gamma$-contractive semigroup $t \mapsto e^{-t A_\sigma}$ on $L^2 ( \Omega )$, that is,
\begin{equation}
\label{semigroup_cont} \| e^{-t A_\sigma } \phi \|_{ L^2 ( \Omega ) } \leq e^{ \gamma t } \| \phi \|_{ L^2 ( \Omega ) } \text{,} \qquad t > 0 \text{,} \quad \phi \in L^2 ( \Omega ) \text{.}
\end{equation}
\end{corollary}

\begin{proof}
Letting $\gamma$ be as in Proposition \ref{prop_resolvent}, then \eqref{resolvent} implies
\[
\| ( \lambda I - A_\sigma )^{-1} f \|_{ L^2 ( \Omega ) } \leq ( \lambda - \gamma )^{-1} \| f \|_{ L^2 ( \Omega ) } \text{,} \qquad f \in L^2 ( \Omega ) \text{,} \quad \lambda > \gamma \text{.}
\]
Moreover, since $\mf{D} ( A_\sigma )$ contains $C^\infty_0 ( \Omega )$, it is dense in $L^2 ( \Omega )$.
Thus, by the above and the Hille--Yosida theorem, we need only show that $A_\sigma$ is closed.

To see this, we consider a sequence $( \phi_k )$ in $\mf{D} ( A_\sigma )$ such that
\begin{equation}
\label{closed_cauchy} \lim_{ k \rightarrow \infty } \phi_k = \phi \text{,} \qquad \lim_{ k \rightarrow \infty } A_\sigma \phi_k = \psi \text{,}
\end{equation}
with both limits in $L^2 ( \Omega )$.
Then, \eqref{elliptic_H2} yields that
\begin{align*}
&\| \nabla_{ -\kappa } \nabla_\kappa ( \phi_k - \phi_l ) ) ] \|_{ L^2 ( \Omega ) } + \| \nabla ( \phi_k - \phi_l ) \|_{ L^2 ( \Omega ) } \\
&\quad \lesssim \| A_\sigma ( \phi_k - \phi_l ) \|_{ L^2 ( \Omega ) } + \| \phi_k - \phi_l \|_{ L^2 ( \Omega ) } \text{,}
\end{align*}
for any $k, l \in \N$.
As the right-hand side of the above converges to zero as $k, l \rightarrow \infty$ by \eqref{closed_cauchy}, then $( \phi_k )$ is a Cauchy sequence in a weighted $H^2$-space, so that
\[
\lim_{ k \rightarrow \infty } \nabla \phi_k = \nabla \phi \text{,} \qquad \lim_{ k \rightarrow \infty } \nabla_{ -\kappa } \nabla_\kappa \phi_k = \nabla_{ -\kappa } \nabla_\kappa \phi \text{.}
\]
The above then implies $\psi = A_\sigma \phi$, proving that $A_\sigma$ is closed.
\end{proof}

\begin{remark}
The assumption $\sigma > -\frac{3}{4}$ ($\kappa > -\frac{1}{2}$) was only used to construct the Neumann operator $\mc{N}_\sigma$.
Other parts of the theory only required $\sigma < \frac{1}{4}$ ($\kappa < \frac{1}{2}$).
\end{remark}

\subsection{Semigroup solutions}

We now turn our attention to the singular heat equations from Problems (OI) and (O).
As in \cite{BZuazua}, we use the semigroup generated in Corollary \ref{semigroup} to build solutions of these problems:

\begin{definition} \label{mild_soln}
Given $u_T \in L^2 ( \Omega )$, $F \in L^2 ( (0, T) \times \Omega )$, we define the associated (semigroup) solution of Problem (OI) to be the map $u \in C^0 ( [ 0, T ]; L^2 ( \Omega ) )$ given by
\begin{equation}
\label{mild_duhamel} u (t) = e^{(T-t) A_\sigma} u_T - \int_t^T e^{(s-t) A_\sigma} F (s) \, ds \text{,} \qquad t \in [ 0, T ] \text{.}
\end{equation}
\end{definition}

Next, we derive the essential regularity properties of these semigroup solutions:

\begin{proposition} \label{wp_mild}
Suppose $u_T \in L^2 ( \Omega )$ and $F \in L^2 ( (0, T) \times \Omega )$, and let $u$ denote the asssociated solution of Problem (OI).
Then,
\begin{equation}
\label{regularity_mild} u \in C^0 ( [ 0, T ]; L^2 ( \Omega ) ) \cap L^2 ( ( 0, T ); H^1_0 ( \Omega ) ) \text{,}
\end{equation}
and $u$ also satisfies the following estimate,
\begin{equation}
\label{energy_mild} \| u \|_{ L^\infty ( [0, T]; L^2 ( \Omega ) ) }^2 + \| \nabla_\kappa u \|_{ L^2 ( ( 0, T ) \times \Omega ) }^2 \lesssim \| u_T \|_{ L^2 ( \Omega ) }^2 + \| F \|_{ L^2 ( ( 0, T ) \times \Omega ) }^2 \text{,}
\end{equation}
where the constant depends on $T$, $\Omega$, $\sigma$, $X$, $V$.
\end{proposition}

\begin{proof}
First, we assume $u_T \in C^\infty_0 ( \Omega )$ and $F \in C^\infty_0 ( ( 0, T ) \times \Omega )$, so that the solution $u$, defined as in \eqref{mild_duhamel}, is smooth and satisfies in particular
\[
u \in C^0 ( [ 0, T ]; \mf{D} ( A_\sigma ) ) \text{,} \qquad \partial_t u \in C^0 ( [ 0, T]; L^2 ( \Omega ) ) \text{.}
\]
Note that Propositions \ref{dirichlet_vanish} and \ref{elliptic} are applicable in this case.

Then, by the fundamental theorem of calculus, \eqref{heat_ex}, and \eqref{mod_op},
\begin{align*}
\| u (T) \|_{ L^2 ( \Omega ) }^2 &= \| u (t) \|_{ L^2 ( \Omega ) }^2 + 2 \int_t^T \int_\Omega u ( F - \nabla_{ -\kappa } \nabla_\kappa u - X \cdot \nabla u - V_y u ) \big|_{t=s} ds \\
&= \| u (t) \|_{ L^2 ( \Omega ) }^2 + 2 \int_t^T \int_\Omega F u |_{t=s} ds + 2 \int_t^T \int_\Omega | \nabla_\kappa u |^2 \big|_{ t = s } ds \\
&\qquad + \int_t^T \int_\Omega ( \nabla \cdot X - 2 V_y ) u^2 \big|_{t=s} ds \text{,}
\end{align*}
for any $t \in [ 0, T )$.
(In the last step, we integrated by parts and used Propositions \ref{dirichlet_vanish} and \ref{elliptic} to eliminate boundary terms; observe also that $\nabla_{ -\kappa } \nabla_\kappa u$ is well-defined due to Proposition \ref{elliptic}.)
By Definition \ref{def.lower_order}, we then obtain
\begin{align*}
\| u (t) \|_{ L^2 ( \Omega ) }^2 + \| \nabla_\kappa u \|_{ L^2 ( ( t, T ) \times \Omega ) }^2 &\lesssim \| u_T \|_{ L^2 ( \Omega ) }^2 + \int_t^T \| F (s) \|_{ L^2 ( \Omega ) } \| u (s) \|_{ L^2 ( \Omega ) } ds \\
&\qquad + \int_t^T \| y^{-1} u (s) \|_{ L^2 ( \Omega ) } \| u (s) \|_{ L^2 ( \Omega ) } ds \text{.}
\end{align*}
Applying \eqref{hardy_main}--\eqref{hardy_norm_equiv} and absorbing terms into the left-hand side yields
\begin{align*}
\| u (t) \|_{ L^2 ( \Omega ) }^2 + \| \nabla_\kappa u \|_{ L^2 ( ( t, T ) \times \Omega ) }^2 &\lesssim \| u_T \|_{ L^2 ( \Omega ) }^2 + \| F \|_{ L^2 ( (0, T) \times \Omega ) }^2 \\
&\qquad + \int_t^T \| u (s) \|_{ L^2 ( \Omega ) }^2 ds \text{,}
\end{align*}
hence \eqref{energy_mild}, for regular $u_T$ and $F$, now follows via Gronwall's inequality.

Finally, for general $u_T \in L^2 ( \Omega )$ and $F \in L^2 ( ( 0, T ) \times \Omega )$, we consider sequences $( u_{ T, k } )$ and $( F_k )$ in $C^\infty_0 ( \Omega )$ and $C^\infty_0 ( ( 0, T ) \times \Omega )$ converging to $u_T$ and $F$ in $L^2 ( \Omega )$ and $L^2 ( ( 0, T ) \times \Omega )$, respectively.
Applying \eqref{energy_mild} to solutions of Problem (OI) arising from the $u_{ T, k }$'s and the $F_k$'s, as well as from their differences, we obtain that \eqref{energy_mild} continues to hold for solutions of Problem (OI) arising from $u_T$ and $F$.
\end{proof}

\begin{proposition} \label{wp_strict}
Suppose $u_T \in H^1_0 ( \Omega )$ and $F \in L^2 ( (0, T) \times \Omega )$, and let $u$ denote the asssociated solution of Problem (OI).
Then,
\begin{equation}
\label{regularity_strict} u \in C^0 ( [ 0, T ]; H^1_0 ( \Omega ) ) \cap H^1 ( ( 0, T ); L^2 ( \Omega ) ) \cap L^2 ( ( 0, T ); \mf{D} ( A_\sigma ) ) \text{,}
\end{equation}
and $u$ satisfies the following almost everywhere on $( 0, T ) \times \Omega$:
\begin{equation}
\label{eq_strict} ( \partial_t + \Delta_\sigma + X \cdot \nabla + V ) u = F \text{.}
\end{equation}
Furthermore, $u$ satisfies the estimate
\begin{align}
\label{energy_strict} &\| u \|_{ L^\infty ( [0, T]; H^1 ( \Omega ) ) }^2 + \| \nabla_{ -\kappa } \nabla_\kappa u \|_{ L^2 ( ( 0, T ) \times \Omega ) }^2 + \| \partial_t u \|_{ L^2 ( ( 0, T ) \times \Omega ) }^2 \\
\notag &\quad \lesssim \| u_T \|_{ H^1 ( \Omega ) }^2 + \| F \|_{ L^2 ( ( 0, T ) \times \Omega ) }^2 \text{,}
\end{align}
with the constant depending on $T$, $\Omega$, $\sigma$, $X$, $V$.
\end{proposition}

\begin{proof}
First, assume $u_T \in C^\infty_0 ( \Omega )$ and $F \in C^\infty_0 ( ( 0, T ) \times \Omega )$, which implies
\[
u \in C^0 ( [ 0, T ]; \mf{D} ( A_\sigma^2 ) ) \text{,} \qquad \partial_t u \in C^0 ( [ 0, T ]; \mf{D} ( A_\sigma ) ) \text{.}
\]
(Recall $\mf{D} ( A_\sigma ^2 )$ was defined in \eqref{AB_2}.)
This then enables the following computation,
\begin{align*}
\| \nabla_\kappa u (T) \|_{ L^2 ( \Omega ) }^2 &= \| \nabla_\kappa u (t) \|_{ L^2 ( \Omega ) }^2 - 2 \int_t^T \int_\Omega \partial_t u \, ( \nabla_{ -\kappa } \cdot \nabla_\kappa u ) \big|_{t=s} ds \\
&= 2 \int_t^T \int_\Omega ( - F + X \cdot \nabla u + V_y \, u ) ( \nabla_{ -\kappa } \cdot \nabla_\kappa u ) \big|_{t=s} ds \\
&\qquad + 2 \int_t^T \int_\Omega | \nabla_{ -\kappa } \cdot \nabla_\kappa u |^2 \big|_{t=s} ds \text{.}
\end{align*}
in which we used \eqref{heat_ex}, \eqref{mod_op}, and an integration by parts (which produce no boundary terms due to Propositions \ref{dirichlet_vanish} and \ref{elliptic}, since $\partial_t u (t) \in H^1_0 ( \Omega )$ for $t \in [ 0, T ]$.)

Rearranging the above, applying \eqref{hardy_norm_equiv}, and recalling also \eqref{energy_mild}, we then have
\begin{align}
\label{approx_strict_0} \| u (t) \|_{ H^1 ( \Omega ) }^2 + \| \nabla_{ -\kappa } \cdot \nabla_\kappa u \|_{ L^2 ( ( t, T ) \times \Omega ) }^2 &\lesssim \| u_T \|_{ H^1 ( \Omega ) }^2 + \| F \|_{ L^2 ( ( t, T ) \times \Omega ) }^2 \\
\notag &\qquad + \int_t^T \int_\Omega \| u (s) \|_{ H^1 ( \Omega ) }^2 \, ds \text{.}
\end{align}
Furthermore, by the heat equation \eqref{heat_ex}, we also obtain the estimate
\begin{align}
\label{approx_strict_t} \| \partial_t u \|_{ L^2 ( ( 0, T ) \times \Omega ) }^2 &\lesssim \| A_\sigma u \|_{ L^2 ( ( 0, T ) \times \Omega ) }^2 + \| F \|_{ L^2 ( ( 0, T ) \times \Omega ) }^2 \\
\notag &\lesssim \| u_T \|_{ H^1 ( \Omega ) }^2 + \| F \|_{ L^2 ( ( 0, T ) \times \Omega ) }^2 \text{.}
\end{align}
Thus, \eqref{energy_strict}, for regular $u_T$ and $F$, now follows from \eqref{elliptic_H2}, \eqref{approx_strict_0}, and \eqref{approx_strict_t}.

Now, for general $u_T \in H^1_0 ( \Omega )$ and $F \in L^2 ( ( 0, T ) \times \Omega )$, we approximate, as in the proof of Proposition \ref{wp_strict}, using data in $C^\infty_0 ( \Omega )$ and $C^\infty_0 ( ( 0, T ) \times \Omega )$, respectively.
From this, we conclude that \eqref{energy_strict} still holds, and that the solution $u$ lies in $L^2 ( ( 0, T ); \mf{D} ( A_\sigma ) )$ and $H^1 ( ( 0, T ); L^2 ( \Omega ) )$.
As each approximate solution to Problem (OI) lies in $C^0 ( [ 0, T ]; H^1_0 ( \Omega ) )$, the same must also hold for $u$, which proves \eqref{regularity_strict}.
Finally, \eqref{regularity_strict} provides enough regularity to make sense of and to verify \eqref{eq_strict}.
\end{proof}

\begin{remark}
In the terminology of \cite{BZuazua, new} and of semigroup theory, $L^2$-solutions of Problem (OI) in the setting of Proposition \ref{wp_mild} are called \emph{mild solutions}, while $H^1_0$-solutions in the setting of Proposition \ref{wp_strict} are called \emph{strict solutions}.
\end{remark}

\begin{remark}
Observe the assumption $\sigma > -\frac{3}{4}$ ($\kappa > -\frac{1}{2}$) was not needed for solving Problems (OI) and (O), or for proving Propositions \ref{wp_mild} and \ref{wp_strict}.
This will only be used to show the existence of and estimates for the Neumann trace, and it will likewise be essential for the upcoming theory of dual solutions.
\end{remark}

\subsection{The Neumann Trace}

Similar to \cite{new}, the next step is to make sense of the Neumann trace in the setting of $H^1_0$-solutions of Problem (OI):

\begin{proposition} \label{neumann_trace}
Let $u_T \in H^1_0 ( \Omega )$ and $F \in L^2 ( (0, T) \times \Omega )$, and let $u$ denote the corresponding solution to Problem (OI).
Then, the Neumann trace $\mc{N}_\sigma u$ is well-defined as an element of $L^2 ( (0, T) \times \Gamma )$.
Furthermore, for $0 < y_0 \ll 1$, we have
\begin{align}
\label{hidden_reg} &\| \mc{N}_\sigma u \|_{ L^2 ( ( 0, T ) \times \Gamma ) }^2 + \sup_{ 0 < \delta < y_0 } \| \eta^\delta [ y^{ 2 \kappa } D_y ( y^{ -\kappa } u ) ] \|_{ L^2 ( ( 0, T ) \times \Gamma ) }^2 \\
\notag &\quad \lesssim \| u_T \|_{ H^1 ( \Omega ) }^2 + \| F \|_{ L^2 ( ( 0, T ) \times \Omega ) }^2 \text{,}
\end{align}
where the constant of the inequality depends on $T, \Omega, \sigma, X, V$.
\end{proposition}

\begin{proof}
As \eqref{regularity_strict} implies $y^{ 2 \kappa } D_y ( y^{ -\kappa } u ) \in L^2 ( ( 0, T ); H^1_{ \mathrm{loc} } ( \Omega ) )$, its traces on level sets of $y$ are well-defined.
Thus, using that $2 \kappa > -1$, we obtain, for $0 < y_1 < y_0 \ll 1$,
\begin{align}
\label{neumann_trace_0} &\| \eta^{ y_0 } y^{ 2 \kappa } D_y ( y^{ -\kappa } u ) - \eta^{ y_1 } y^{ 2 \kappa } D_y ( y^{ -\kappa } u ) \|_{ L^2 ( ( 0, T ) \times \Gamma ) } \\
\notag &\quad = \int_{ ( 0, T ) \times \Gamma } \bigg( \int_{ y_1 }^{ y_0 } \eta^s D_y [ y^{ 2 \kappa } D_y ( y^{ -\kappa } u ) ] \, ds \bigg)^2 \\
\notag &\quad \leq \int_{ y_1 }^{ y_0 } y^{ 2 \kappa } dy \cdot ( \| \nabla_{ -\kappa } \nabla_\kappa \phi \|_{ L^2 ( ( 0, T ) \times \Omega ) }^2 + \| \phi \|_{ L^2 ( ( 0, T ); H^1 ( \Omega ) }^2 ) \\
\notag &\quad \lesssim y_0^{ 1 + 2 \kappa } ( \| \nabla_{ -\kappa } \nabla_\kappa \phi \|_{ L^2 ( ( 0, T ) \times \Omega ) }^2 + \| \phi \|_{ L^2 ( ( 0, T ); H^1 ( \Omega ) }^2 ) \text{,}
\end{align}
By \eqref{energy_strict}, the right-hand side of the above vanishes as $y_0 \searrow 0$, and it follows that $\mc{N}_\sigma u$ is well-defined as an element of $L^2 ( ( 0, T ) \times \Gamma )$.

A similar estimate using also a cutoff in $y$ as in \eqref{cutoff_y} yields
\[
\| \mc{N}_\sigma u \|_{ L^2 ( ( 0, T ) \times \Gamma ) }^2 \lesssim \| \nabla_{ -\kappa } \nabla_\kappa \phi \|_{ L^2 ( ( 0, T ) \times \Omega ) }^2 + \| \phi \|_{ L^2 ( ( 0, T ); H^1 ( \Omega ) }^2 \text{,}
\]
and an application of \eqref{energy_strict} proves the bound for $\mc{N}_\sigma u$ in \eqref{hidden_reg}.
The corresponding estimates for the $\eta^\delta y^{ 2 \kappa } D_y ( y^{ -\kappa } u )$'s in \eqref{hidden_reg} proceed in the same manner as the above, except that one controls from $y = \delta > 0$ rather than from $y = 0$.
\end{proof}

\begin{proposition} \label{neumann_limits}
Let $u_T \in H^1_0 ( \Omega )$ and $F \in L^2 ( (0, T) \times \Omega )$, and let $u$ denote the corresponding solution to Problem (OI).
Then, for any $0 < y_0 \ll 1$,
\begin{align}
\label{hidden_dirichlet} \big\| \eta^{ y_0 } [ y^{ 2 \kappa } D_y ( y^{ -\kappa } u ) ] - \mc{N}_\sigma u \big\|_{ L^2 ( ( 0, T ) \times \Gamma ) }^2 &\lesssim y_0^{ 1 + 2 \kappa } ( \| u_T \|_{ H^1 ( \Omega ) }^2 + \| F \|_{ L^2 ( ( 0, T ) \times \Omega ) }^2 ) \text{,} \\
\notag \big\| \eta^{ y_0 } ( y^{ -1 + \kappa } u ) - \tfrac{1}{ 1 - 2 \kappa } \mc{N}_\sigma u \big\|_{ L^2 ( ( 0, T ) \times \Gamma ) }^2 &\lesssim y_0^{ 1 + 2 \kappa } ( \| u_T \|_{ H^1 ( \Omega ) }^2 + \| F \|_{ L^2 ( ( 0, T ) \times \Omega ) }^2 ) \text{,} \\
\notag \big\| \eta^{ y_0 } ( y^\kappa D_y u ) - \tfrac{ 1 - \kappa }{ 1 - 2 \kappa } \mc{N}_\sigma u \big\|_{ L^2 ( ( 0, T ) \times \Gamma ) }^2 &\lesssim y_0^{ 1 + 2 \kappa } ( \| u_T \|_{ H^1 ( \Omega ) }^2 + \| F \|_{ L^2 ( ( 0, T ) \times \Omega ) }^2 ) \text{,}
\end{align}
with the constants depending on $T$, $\Omega$, $\sigma$, $X$, $V$.
\end{proposition}

\begin{proof}
The first part of \eqref{hidden_dirichlet} follows immediately from the estimate \eqref{neumann_trace_0}, once we take $y_1 \searrow 0$ and then apply \eqref{energy_strict}.
For the second part, we first note
\begin{align*}
&\big\| \eta^{ y_0 } y^{ \kappa - 1 } u - \tfrac{1}{ 1 - 2 \kappa } \mc{N}_\sigma u \big\|_{ L^2 ( ( 0, T ) \times \Omega ) }^2 \\
&\quad = \int_{ ( 0, T ) \times \Gamma } \bigg[ y_0^{ 2 \kappa - 1 } \int_0^{ y_0 } \eta^s D_y ( y^{ -\kappa } u ) \, ds - \tfrac{1}{ 1 - 2 \kappa } \mc{N}_\sigma u \bigg]^2 \\
&\quad = \int_{ ( 0, T ) \times \Gamma } \bigg( y_0^{ 2 \kappa - 1 } \int_0^{ y_0 } s^{ -2 \kappa } [ \eta^s y^{ 2 \kappa } D_y ( y^{ -\kappa } u ) - \mc{N}_\sigma u \big] \, ds \bigg)^2 \text{,}
\end{align*}
where we used that $\mc{D}_\sigma u = 0$ from Proposition \ref{dirichlet_vanish}.
By Minkowski's inequality,
\begin{align*}
&\big\| \eta^{ y_0 } y^{ \kappa - 1 } u - \tfrac{1}{ 1 - 2 \kappa } \mc{N}_\sigma u \big\|_{ L^2 ( ( 0, T ) \times \Omega ) }^2 \\
&\quad \lesssim y_0^{ 2 \kappa - 1 } \int_0^{ y_0 } s^{ -2 \kappa } \| \eta^s y^{ 2 \kappa } D_y ( y^{ -\kappa } u ) - \mc{N}_\sigma u \|_{ L^2 ( ( 0, T ) \times \Gamma ) } \, ds \\
&\quad \lesssim \sup_{ 0 < s < y_0 } \| \eta^s y^{ 2 \kappa } D_y ( y^{ -\kappa } u ) - \mc{N}_\sigma u \|_{ L^2 ( ( 0, T ) \times \Gamma ) } \text{,}
\end{align*}
where we also noted that $2 \kappa > -1$ in the last step.
The second estimate in \eqref{hidden_dirichlet} now follows from the above and from the first part.
Finally, the third limit of \eqref{hidden_dirichlet} is now an immediate consequence of the first two, since
\[
y^\kappa D_y u = y^{ 2 \kappa } D_y ( y^{ -\kappa } u ) + \kappa y^{ \kappa - 1 } u \text{.} \qedhere
\]
\end{proof}

The following technical observation will also be needed to deal with an irregular boundary term arising in our Carleman estimates:

\begin{proposition} \label{stupid_boundary}
Let $u_T \in H^1_0 ( \Omega )$ and $F \in L^2 ( (0, T) \times \Omega )$, and let $u$ denote the corresponding solution to Problem (OI).
Also, let $w \in C^1 ( [ 0, T ] \times \Omega )$ satisfy
\begin{equation}
\label{stupid_weight} \| w \|_{ L^\infty ( ( 0, T ) \times \Omega ) } + \| \partial_t w \|_{ L^\infty ( ( 0, T ) \times \Omega ) } + \| y^{ -2 p } \nabla w \|_{ L^\infty ( ( 0, T ) \times \Omega ) } < \infty
\end{equation}
for some $p < 0$ satisfying $2 p - \kappa > -\frac{1}{2}$.
Then, for any $0 < \delta \ll 1$:
\begin{itemize}
\item The following quantities are well-defined in the trace sense:
\begin{align}
\label{stupid_trace} B_0 ( u_T, F; \delta ) &:= \int_{ ( 0, T ) \times \{ y = \delta \} } w \cdot \partial_t ( y^{ -\kappa } u ) \cdot y^{ -1 + \kappa } u \text{,} \\
\notag B_1 ( u_T, F; \delta ) &:= \int_{ ( 0, T ) \times \{ y = \delta \} } w \cdot \partial_t ( y^{ -\kappa } u ) \cdot y^{ 2 \kappa } D_y ( y^{ -\kappa } u ) \text{.}
\end{align}

\item The following bound holds, with constant depending on $T$, $\Omega$, $\sigma$, $X$, $V$, $p$:
\begin{equation}
\label{stupid_estimate} | B_0 ( u_T, F; \delta ) | + | B_1 ( u_T, F; \delta ) | \lesssim \| u_T \|_{ H^1 ( \Omega ) }^2 + \| F \|_{ L^2 ( ( 0, T ) \times \Omega ) }^2 \text{.}
\end{equation}
\end{itemize}
Furthermore, the following boundary limits hold:
\begin{equation}
\label{stupid_limit} \lim_{ \delta \searrow 0 } B_0 ( u_T, F; \delta ) = 0 \text{,} \qquad \lim_{ \delta \searrow 0 } B_1 ( u_T, F; \delta ) = 0 \text{.}
\end{equation}
\end{proposition}

\begin{proof}
First, let us assume $u_T \in C^\infty_0 ( \Omega )$ and $F \in C^\infty_0 ( ( 0, T ) \times \Omega )$, which implies
\[
\partial_t u \in C^0 ( [ 0, T ]; H^1_0 ( \Omega ) ) \text{,} \qquad y^{ 2 \kappa } D_y ( y^{ -2 \kappa } u ), y^{ -1 + \kappa } u \in L^2 ( ( 0, T ); H^1_{ \textrm{loc} } ( \Omega ) ) \text{.}
\]
In particular, $B_0 ( u_T, F; \delta )$ and $B_1 ( u_T, F; \delta )$ in \eqref{stupid_trace} are well-defined.
Moreover,
\begin{align}
\label{eql.stupid_1} | B_0 ( u_T, F; \delta ) | &= \bigg| \int_{ ( 0, T ) \times \{ y = \delta \} } w y^{-1} \, \partial_t ( u^2 ) \, \bigg| \\
\notag &= \sup_{ 0 \leq s \leq T } \bigg| \int_{ \{ y = \delta \} } ( w y^{-1} \, u^2 ) |_{ t = s } \, \bigg| + \bigg| \int_{ ( 0, T ) \times \{ y = \delta \} } \partial_t w \, y^{-1} u^2 \, \bigg| \\
\notag &\lesssim \| u_T \|_{ H^1 ( \Omega ) }^2 + \| F \|_{ L^2 ( ( 0, T ) \times \Omega ) }^2 \text{,}
\end{align}
where in the last step, we applied \eqref{hardy_norm_equiv}, \eqref{dirichlet_extra}, \eqref{energy_strict}, and \eqref{stupid_weight}.

For $B_1$, notice that by multiplying $w$ by a regular function (corresponding to a change of volume form), it suffices to bound instead the quantity
\begin{equation}
\label{eql.stupid_0} \bar{B}_1 ( u_T, F; \delta ) := \int_{ ( 0, T ) \times \Gamma } \eta^{ \delta } [ w \cdot \partial_t ( y^{ -\kappa } u ) \cdot y^{ 2 \kappa } D_y ( y^{ -\kappa } u ) ] \text{.}
\end{equation}
Letting $\chi$ be a cutoff function defined as in \eqref{cutoff_y}, with $\delta \ll y_0 \ll 1$, we have
\begin{align}
\label{eql.stupid_10} | \bar{B}_1 ( u_T, F; \delta ) | &= \bigg| \int_{ ( 0, T ) \times \Gamma } \int_\delta^{ 2 y_0 } \eta^s D_y [ \chi w \cdot \partial_t ( y^{ -\kappa } u ) \cdot y^{ 2 \kappa } D_y ( y^{ -\kappa } u ) ] \, ds \, \bigg| \\
\notag &\leq \int_{ ( 0, T ) \times \Gamma } \int_\delta^{ 2 y_0 } \eta^s | \partial_t u | [ | \nabla_{ -\kappa } \nabla_\kappa u | + | y^{ 2 p + \kappa } D_y ( y^{ -\kappa } u ) | ] \, ds \\
\notag &\qquad + \bigg| \frac{1}{2} \int_{ ( 0, T ) \times \Gamma } \int_\delta^{ 2 y_0 } \eta^s [ \chi w \cdot \partial_t ( | \nabla_\kappa u |^2 ) ] \, ds \, \bigg| \\
\notag :\!\!&= I_1 + I_2 + I_3 \text{,}
\end{align}
where we also made use of \eqref{stupid_weight} to obtain $I_1$ and $I_2$.

For $I_3$, we integrate by parts and then apply \eqref{hardy_norm_equiv}, \eqref{energy_strict}, and \eqref{stupid_weight}:
\begin{align}
\label{eql.stupid_11} I_3 &\lesssim \sup_{ 0 \leq s \leq T } \int_\Omega | \nabla_\kappa u |^2 + \int_{ ( 0, T ) \times \Omega } | \nabla_\kappa u |^2 \\
\notag &\lesssim \| u_T \|_{ H^1 ( \Omega ) }^2 + \| F \|_{ L^2 ( ( 0, T ) \times \Omega ) }^2 \text{.}
\end{align}
In addition, a direct application of \eqref{energy_strict} yields
\begin{equation}
\label{eql.stupid_12} I_1 \lesssim \| u_T \|_{ H^1 ( \Omega ) }^2 + \| F \|_{ L^2 ( ( 0, T ) \times \Omega ) }^2 \text{.}
\end{equation}
For $I_2$, we apply the H\"older inequality to bound
\begin{align*}
I_2^2 &\lesssim \| \partial_t u \|_{ L^2 ( ( 0, T ) \times \Omega ) }^2 \int_{ ( 0, T ) \times \Gamma } \int_0^{ 2 y_0 } | \eta^s y^{ 2 p - \kappa } y^{ 2 \kappa } D_y ( y^{ -\kappa } u ) |^2 \, ds \\
&\lesssim \| \partial_t u \|_{ L^2 ( ( 0, T ) \times \Omega ) }^2 \int_0^{ 2 y_0 } s^{ 4 p - 2 \kappa } \, ds \sup_{ 0 < s < 2 y_0 } \int_{ ( 0, T ) \times \Gamma } | \eta^s y^{ 2 \kappa } D_y ( y^{ -\kappa } u ) |^2 \text{.}
\end{align*}
Combining the above with \eqref{energy_strict}, \eqref{hidden_reg}, and the assumption $4 p - 2 \kappa > -1$ yields
\begin{equation}
\label{eql.stupid_13} I_3 \lesssim \| u_T \|_{ H^1 ( \Omega ) }^2 + \| F \|_{ L^2 ( ( 0, T ) \times \Omega ) }^2 \text{.}
\end{equation}

Now, from \eqref{eql.stupid_1}, \eqref{eql.stupid_10}--\eqref{eql.stupid_13}, we conclude that \eqref{stupid_estimate} holds for $u_T \in C^\infty_0 ( \Omega )$ and $F \in C^\infty_0 ( ( 0, T ) \times \Omega )$.
An approximation argument based on the bound \eqref{stupid_estimate} then yields that both $B_0 ( u_T, F; \delta )$ and $B_1 ( u_T, F; \delta )$ can be continuously extended to the general case $u_T \in H^1_0 ( \Omega )$ and $F \in L^2 ( ( 0, T ) \times \Omega )$, and that \eqref{stupid_estimate} still holds in this setting.
In particular, this completes the proofs of \eqref{stupid_trace} and \eqref{stupid_estimate}.

It remains only to show \eqref{stupid_limit}.
For this, we first observe that by similar estimates as above, but now set between two level sets of $y$, we derive that
\[
| B_0 ( u_T, F; y_0 ) - B_0 ( u_T, F; y_1 ) | + | B_1 ( u_T, F; y_0 ) - B_1 ( u_T, F; y_1 ) | \rightarrow 0
\]
as $y_0, y_1 \searrow 0$, for all $u_T \in H^1_0 ( \Omega )$, $F \in L^2 ( ( 0, T ) \times \Omega )$.
Thus, the boundary limits
\begin{equation}
\label{eql.stupid_20} B_0 ( u_T, F; 0 ) := \lim_{ \delta \searrow 0 } B_0 ( u_T, F; \delta ) \text{,} \qquad B_1 ( u_T, F; 0 ) := \lim_{ \delta \searrow 0 } B_1 ( u_T, F; \delta )
\end{equation}
are well-defined, and \eqref{stupid_estimate} implies the inequality
\begin{equation}
\label{eql.stupid_21} | B_0 ( u_T, F; 0 ) | + | B_1 ( u_T, F; 0 ) | \lesssim \| u_T \|_{ H^1 ( \Omega ) }^2 + \| F \|_{ L^2 ( ( 0, T ) \times \Omega ) }^2 \text{.}
\end{equation}

Finally, if $u_T \in C^\infty_0 ( \Omega )$ and $F \in C^\infty_0 ( ( 0, T ) \times \Omega )$, then Proposition \ref{dirichlet_vanish} implies $\mc{D}_\sigma ( \partial_t u ) \equiv 0$, hence Propositions \ref{neumann_trace} and \ref{neumann_limits} yield
\[
B_0 ( u_T, F; 0 ) = B_1 ( u_T, F; 0 ) = 0 \text{.}
\]
Since $C^\infty_0 ( \Omega )$ and $C^\infty_0 ( ( 0, T ) \times \Omega )$ are dense in $H^1_0 ( \Omega )$ and $L^2 ( ( 0, T ) \times \Omega )$, then an approximation argument using \eqref{eql.stupid_21} yields \eqref{stupid_limit} for general $u_T \in H^1_0 ( \Omega )$ and $F \in L^2 ( ( 0, T ) \times \Omega )$ as well, which completes the proof of \eqref{stupid_limit}.
\end{proof}

\subsection{Dual solutions}

Finally, we treat the dual theory of solutions for Problem (C).
As in \cite{new}, the first step is to define solutions at $H^{-1}$-regularity:

\begin{definition} \label{transposition}
Given $v_0 \in H^{-1} ( \Omega )$ and $v_d \in L^2 ( (0, T) \times \Ga )$, we call
\[
v \in C^0 ( [ 0, T ]; H^{-1} ( \Om ) ) \cap L^2 ( ( 0, T ) \times \Omega )
\]
a \emph{dual} (or \emph{transposition}) \emph{solution} of Problem (C) iff for any $F \in L^2 ( ( 0, T ) \times \Om )$,
\begin{equation}
\label{weak_soln} \int_{ ( 0, T ) \times \Omega } F v = -\int_\Omega u (0) \, v_0 + \int_{ ( 0, T ) \times \Gamma } \mc{N}_\sigma u \, v_d \text{,}
\end{equation}
where $u$ is the solution to Problem (OI) with $F$ as above, $u_T \equiv 0$, and
\begin{equation}
\label{dual_XV} X := -Y \text{,} \qquad V := W - \nabla \cdot Y \text{.}
\end{equation}
\end{definition}

\begin{remark}
\eqref{dual_XV} ensures that \eqref{heat_ctl} is the adjoint equation to \eqref{heat_ex}.
Moreover, note that if $( Y, W ) \in \mc{Z}_0$, then $( X, V )$ from \eqref{dual_XV} also lies in $\mc{Z}_0$.
\end{remark}

Substantial revisions are needed to extend the theory of dual solutions in \cite{new} to our more general setting of $y \in C^2 ( \Omega )$ and $( Y, W ) \in \mc{Z}_0$.
Thus, we provide a new and more detailed development of the key regularity properties below.

\begin{proposition} \label{wp_weak}
For any $v_0 \in H^{-1} ( \Omega )$ and $v_d \in L^2 ( (0, T) \times \Ga )$, there exists a unique weak solution $v$ of Problem (C).
In addition, $v$ satisfies the bound
\begin{equation}
\label{energy_weak} \| v \|_{ L^\infty ( [ 0, T ]; H^{-1} ( \Omega ) ) }^2 + \| v \|_{ L^2 ( ( 0, T ) \times \Omega ) }^2 \lesssim \| v_0 \|_{ H^{-1} ( \Omega ) }^2 + \| v_d \|_{ L^2 ( ( 0, T ) \times \Gamma ) }^2 \text{,}
\end{equation}
where the constant depends only on $T, \Omega, \sigma, Y, W$.

Furthermore, if $u_T \in H^1_0 ( \Omega )$, and if $u$ is the corresponding solution to Problem (O), with $( X, V )$ as in \eqref{dual_XV}, then the following identity holds:
\begin{equation}
\label{weak_id} \int_\Omega [ u_T \, v (T) - u (0) \, v_0 ] + \int_{ ( 0, T ) \times \Gamma } \mc{N}_\sigma u \, v_d = 0 \text{.}
\end{equation}
\end{proposition}

\begin{proof}
Define the linear functional $S: L^2 ( ( 0, T ) \times \Omega ) \rightarrow \R$ by
\[
S F := -\int_\Omega u (0) \, v_0 + \int_{ ( 0, T ) \times \Gamma } \mc{N}_\sigma u \, v_d \text{,}
\]
with $u$ being the solution to Problem (OI) with the above $F$, with $u_T \equiv 0$, and with $( X, V )$ as in \eqref{dual_XV}.
By \eqref{energy_strict} and \eqref{hidden_reg}, we have
\begin{align*}
| S F | &\lesssim \| u (0) \|_{ H^1 ( \Omega ) } \| v_0 \|_{ H^{-1} ( \Omega ) } + \| \mc{N}_\sigma u \|_{ L^2 ( ( 0, T ) \times \Gamma ) } \| v_d \|_{ L^2 ( ( 0, T ) \times \Gamma ) } \\
\notag &\lesssim ( \| v_0 \|_{ H^{-1} ( \Omega ) } + \| v_d \|_{ L^2 ( ( 0, T ) \times \Gamma ) } ) \| F \|_{ L^2 ( ( 0, T ) \times \Omega ) } \text{,}
\end{align*}
hence $S$ is bounded.
The Riesz theorem yields a unique $v \in L^2 ( ( 0, T ) \times \Omega )$ with
\[
\int_{ ( 0, T ) \times \Omega } F v = S F \text{,}
\]
hence $v$ satisfies the desired identity \eqref{weak_soln}, as well as the estimate
\begin{equation}
\label{energy_weak_pre} \| v \|_{ L^2 ( ( 0, T ) \times \Omega ) }^2 \lesssim \| v_0 \|_{ H^{-1} ( \Omega ) }^2 + \| v_d \|_{ L^2 ( ( 0, T ) \times \Gamma ) }^2 \text{.}
\end{equation}
As a result, it remains only to obtain the $C^0 ( [ 0, T ]; H^{-1} ( \Omega ) )$-regularity for $v$, the $L^\infty ( [ 0, T ]; H^{-1} ( \Omega ) )$-estimate for $v$ in \eqref{energy_weak}, and the identity \eqref{weak_id}.

First, consider the special case of regular data in Problem (C):
\begin{equation}
\label{regular_C} v_0 \in C^\infty_0 ( \Omega ) \text{,} \qquad v_d \in C^\infty_0 ( ( 0, T ) \times \Gamma ) \text{.}
\end{equation}
Extend $v_d$ into a function $G_d \in C^2 ( ( 0, T ) \times \bar{\Omega} )$ satisfying, near $( 0, T ) \times \Gamma$,
\begin{equation}
\label{G_d} G_d |_{ ( 0, T ) \times \Gamma } = v_d \text{,} \qquad \nabla d_\Gamma \cdot \nabla G_d = 0 \text{.}
\end{equation}
A direct computation then yields
\begin{align}
\label{duality_1} ( -\partial_t + B_\sigma ) ( y^\kappa G_d ) &= y^{ -1 + \kappa } ( 2 \kappa \nabla y \cdot \nabla G_d + \kappa Y \cdot \nabla y \, G_d + y W_y \, G_d ) \\
\notag &\qquad + y^\ka ( \partial_t G_d + \Delta G_d + Y \cdot \nabla G_d ) \text{,} \\
\notag :\!\!&= y^{-1} M \text{,}
\end{align}
Note in particular that $M \in L^2 ( ( 0, T ) \times \Omega )$, since $\kappa > -\frac{1}{2}$.

For any sufficiently large $l \in \N$, we let $v_{ h, l }$ be the (semigroup) solution of
\begin{align}
\label{vh} ( -\partial_t + B_\sigma ) v_{ h, l } = - y^{-1} M \chi_{ \{ y > l^{-1} \} } &\quad \text{on $( 0, T ) \times \Omega$,} \\
\notag v_{ h, l } ( 0 ) = v_0 &\quad \text{on $\Omega$,} \\
\notag v_{ h, l } = 0 &\quad \text{on $( 0, T ) \times \Gamma$,}
\end{align}
where $\chi_B$ denotes the characteristic function on $B$.
(Note semigroup solutions are analogously defined for forward heat equations, and $- y^{-1} M \chi_l \in L^2 ( ( 0, T ) \times \Omega )$.)
Then, by \eqref{duality_1} and \eqref{vh}, the function $v_l := v_{ h, l } + y^\kappa G_d$ solves
\begin{align}
\label{v_dir} ( -\partial_t + B_\sigma ) v_l = y^{-1} M \chi_{ \{ y \leq l^{-1} \} } &\quad \text{on $( 0, T ) \times \Omega$,} \\
\notag v_l ( 0 ) = v_0 &\quad \text{on $\Omega$,} \\
\notag \mc{D}_\sigma v_l = v_d &\quad \text{on $( 0, T ) \times \Gamma$.}
\end{align}
Furthermore, by Proposition \ref{neumann_trace} and \eqref{G_d}, we have
\begin{equation}
\label{v_neumann} \mc{N}_\sigma v_l = \mc{N}_\sigma v_{ l, h } \in L^2 ( ( 0, T ) \times \Gamma ) \text{.}
\end{equation}

Now, given any $u_T \in C^\infty_0 ( \Omega )$ and $F \in C^\infty_0 ( ( 0, T ) \times \Omega )$, the corresponding solution $u$ of Problem (OI), with $( X, Y )$ as in \eqref{dual_XV}, satisfies, via integrations by parts,
\begin{align}
\label{duality_10} \int_{ ( 0, T ) \times \Omega } F v_l &= \int_{ ( 0, T ) \times \Omega } ( \partial_t u + A_\sigma u ) v_l \\
\notag &= \int_{ ( 0, T ) \times \Omega } u ( -\partial_t v_l + B_\sigma v_l ) + \int_\Omega [ u_T \, v_l (T) - u (0) \, v_0 ] \\
\notag &\qquad + \int_{ ( 0, T ) \times \Gamma } ( \mc{N}_\sigma u \mc{D}_\sigma v_l - \mc{D}_\sigma u \mc{N}_\sigma v_l ) \\
\notag &= \int_\Omega [ u_T \, v_l (T) - u (0) \, v_0 ] + \int_{ ( 0, T ) \times \Gamma } \mc{N}_\sigma u \, v_d \\
\notag &\qquad + \int_{ ( 0, T ) \times \Omega } y^{-1} M u \chi_{ \{ y \leq l^{-1} \} } \text{.}
\end{align}
(Note all the integrals above exist due to the extra regularity of $v_0$ and $u_T$.
We also recalled \eqref{v_dir}--\eqref{v_neumann} and noted that one boundary term vanishes since $\mc{D}_\sigma u = 0$.)

Noting from \eqref{v_dir} that $v_l - v_m$, for any $1 \ll l < m$, satisfies
\begin{align*}
( -\partial_t + B_\sigma ) ( v_l - v_m ) &= y^{-1} M \chi_{ \{ m^{-1} < y \leq l^{-1} \} } \text{,} \\
( v_l - v_m )(0) &= 0 \text{,} \\
\mc{D}_\sigma ( v_l - v_m ) &= 0 \text{,}
\end{align*}
and applying the analogue of the proof of Proposition \ref{wp_mild} for the forward heat equation to the above, we obtain the following identity for any $t \in [ 0, T ]$:
\begin{align*}
&\tfrac{1}{2} \| ( v_l - v_m ) (t) \|_{ L^2 ( \Omega ) }^2 + \int_{ ( 0, t ) \times \Omega } | \nabla_\kappa ( v_l - v_m ) |^2 \\
&\quad = - \int_{ ( 0, t ) \times \Omega } y^{-1} ( v_l - v_m ) [ M \chi_{ \{ m^{-1} < y \leq l^{-1} \} } + ( y Y \cdot \nabla + y W_y ) ( v_l - v_m ) ] \text{.}
\end{align*}
Applying \eqref{hardy_main}--\eqref{hardy_norm_equiv} and then the Gronwall inequality to the above (see again the proof of Proposition \ref{wp_mild}) results in the bound
\[
\| v_l - v_m \|_{ L^\infty ( [ 0, T ]; L^2 ( \Omega ) ) } \lesssim \int_{ ( 0, T ) \times \{ m^{-1} < y \leq l^{-1} \} } M^2 \text{.}
\]
As $M \in L^2 ( ( 0, T ) \times \Omega )$, then $\{ v_l \}$ is a Cauchy sequence and thus converges to some $v_\ast$ in $C^0 ( [ 0, T ]; L^2 ( \Omega ) )$.
Letting $l \nearrow \infty$ in \eqref{duality_10} (and recalling \eqref{hardy_norm_equiv}), we have
\begin{equation}
\label{duality_20} \int_{ ( 0, T ) \times \Omega } F v_\ast = \int_\Omega [ u_T \, v_\ast (T) - u (0) \, v_0 ] + \int_{ ( 0, T ) \times \Gamma } \mc{N}_\sigma u \, v_d \text{.}
\end{equation}
Moreover, by an approximation, \eqref{duality_20} continues to hold even when $u_T \in H^1_0 ( \Omega )$.

Setting $u_T \equiv 0$ in \eqref{duality_20} and recalling that \eqref{weak_soln} uniquely determines $v$, we obtain $v_\ast = v$.
Setting $F \equiv 0$ in \eqref{duality_20} instead yields the identity \eqref{weak_id}.
Repeating the above for solutions $u$ of Problem (O) over smaller intervals yields
\begin{equation}
\label{duality_21} \int_\Omega [ u ( t_+ ) \, v ( t_+ ) - u ( t_- ) \, v ( t_- ) ] + \int_{ ( t_-, t_+ ) \times \Gamma } \mc{N}_\sigma u \, v_d = 0 \text{,}
\end{equation}
for any $0 \leq t_- < t_+ \leq T$.
Applying \eqref{energy_strict} and \eqref{hidden_reg} to \eqref{duality_21}, we have
\begin{align*}
\left| \int_\Omega u ( t_+ ) \, v ( t_+ ) \right| &\leq \| u (0) \|_{ H^1 ( \Omega ) } \| v_0 \|_{ H^{-1} ( \Omega ) } + \| \mc{N}_\sigma u \|_{ L^2 ( ( 0, T ) \times \Gamma ) } \| v_d \|_{ L^2 ( ( 0, T ) \times \Gamma ) } \\
&\lesssim \| u ( t_+ ) \|_{ H^1 ( \Omega ) } [ \| v_0 \|_{ H^{-1} ( \Omega ) } + \| v_d \|_{ L^2 ( ( 0, T ) \times \Gamma ) } ] \text{,}
\end{align*}
so varying $u ( t_+ )$ results in the $L^\infty ( [ 0, T ]; H^{-1} ( \Omega ) )$-estimate for $v$ in \eqref{energy_weak}:
\[
\| v \|_{ L^\infty ( [ 0, T ]; H^{-1} ( \Omega ) ) }^2 \lesssim \| v_0 \|_{ H^{-1} ( \Omega ) }^2 + \| v_d \|_{ L^2 ( ( 0, T ) \times \Gamma ) }^2\text{.}
\]
Similarly, \eqref{duality_21} also yields that $v \in C^0 ( [ 0, T ]; H^{-1} ( \Omega ) )$, hence by Definition \ref{transposition}, we have shown that $v$ is indeed a weak solution to Problem (C).

At this point, we have proved the proposition in the special case of regular data \eqref{regular_C}.
The proof is now completed via standard approximations, which show the proposition still holds for $v_0 \in H^{-1} ( \Omega )$ and $v_d \in L^2 ( ( 0, T ) \times \Gamma )$.
\end{proof}

\section{The Local Carleman Estimate} \label{S.carleman}

In this section, we prove our local Carleman estimate for solutions to Problem (O).
Throughout, we will remain with the notations and conventions introduced in Section \ref{S.wp}.
In particular, we let $y$ and $D_y$ be as in Definition \ref{bdf}.

\begin{remark}
As we will work exclusively near $\Gamma$, then in practice, $y$ coincides with $d_\Gamma$.
However, we will write $y$ to maintain consistency with Section \ref{S.wp} and \cite{new}.
\end{remark}

For our Carleman estimate, we will require special local coordinates near $\Gamma$:

\begin{definition} \label{def.w}
Given $x_0 \in \Gamma$ and sufficiently small $\varepsilon > 0$, we define
\begin{equation}
\label{eq.Carleman_regions} B^\Omega_\varepsilon ( x_0 ) := \Omega \cap B_\varepsilon ( x_0 ) \text{,} \qquad B^\Gamma_\varepsilon ( x_0 ) := \Gamma \cap B_\varepsilon ( x_0 ) \text{,}
\end{equation}
where $B_\varepsilon ( x_0 )$ is the open ball in $\R^n$ about $x_0$ of radius $\varepsilon$.
Also:
\begin{itemize}
\item We fix $C^2$-bounded coordinates $w := ( w^1, \dots, w^{n-1} )$ on $B^\Gamma_\varepsilon ( x_0 )$, with
\begin{equation}
\label{eq.w} w ( x_0 ) := 0 \text{.}
\end{equation}

\item We then extend $w$ into $C^2$-bounded coordinates $( y, w )$ on $B^\Omega_\varepsilon ( x_0 )$, such that $w$ is constant along the integral curves of the gradient $\nabla y$.

\item Furthermore, we define the following for convenience:
\begin{equation}
\label{eq.w_norm} | w |^2 := \sum_{ l = 1 }^{n-1} w_l^2 \text{.}
\end{equation}
\end{itemize}
\end{definition}

\begin{remark}
More concretely, one can take $w$ on $B^\Gamma_\varepsilon ( x_0 )$ to be normal coordinates centered at $x_0$.
However, we will not need this specificity in our analysis.
\end{remark}

\begin{remark}
Note the Euclidean metric in these $( y, w )$-coordinates is given by
\[
d y^2 + h_{kl} ( y, w ) \, d w^k d w^l \text{,}
\]
for some components $h_{kl}$.
Thus, the following relations hold on $B^\Omega_\varepsilon ( x_0 )$:
\begin{equation}
\label{eq.ortho} | \nabla y |^2 = 1 \text{,} \qquad \nabla y \cdot \nabla w^i = 0 \text{,} \quad 1 \leq i < n \text{.}
\end{equation}
Furthermore, differentiating the first part of \eqref{eq.ortho} yields
\begin{equation}
\label{eq.hessian} \nabla y \cdot \nabla^2 y \cdot \nabla y = 0 \text{.}
\end{equation}
\end{remark}

Finally, we define the Carleman weight that we will use in this section:

\begin{definition} \label{def.f}
Let $x_0 \in \Gamma$, let $p \in ( -\frac{1}{2}, 0 )$, and let $\varepsilon_0 > 0$ be sufficiently small.
We then define the Carleman weight function $f := f_p$ on $( 0, T ) \times B^\Omega_{ \varepsilon_0 } ( x_0 )$ as follows:
\begin{equation}
\label{eq.f} f := \theta (t) \, \big( \tfrac{1}{ 1 + 2p } y^{ 1 + 2 p } + | w |^2 \big) \text{,} \qquad \theta (t) := \tfrac{1}{ t (T - t) } \text{.}
\end{equation}
\end{definition}

\begin{remark}
Observe $f$ in \eqref{eq.f} extends continuously to $( 0, T ) \times B^\Gamma_{ \varepsilon_0 } ( x_0 )$.
Moreover, $f$ is everywhere non-negative and vanishes only at $( 0, T ) \times \{ x_0 \}$.
\end{remark}

\subsection{The pointwise estimate}

The most substantial component of our local Carleman estimate is captured in the following pointwise inequality:

\begin{theorem}[Pointwise Carleman estimate] \label{thm.pointC}
Let $p \in ( -\frac{1}{2}, 0 )$ satisfy
\begin{equation}
\label{eq.ass_p} \begin{cases} p \leq \kappa & \sigma < 0 \text{,} \\ |p| \ll \frac{1}{4} - \sigma & \sigma > 0 \text{.} \end{cases}
\end{equation}
Moreover, fix $x_0 \in \Gamma$, and let $f := f_p, \theta$ be as in Definition \ref{def.f}.
Then, there exist constants $C, \bar{C}, \lambda_0, \varepsilon_0 > 0$ (depending on $T, \Omega, \sigma, p$) so that for any $\lambda \geq \lambda_0$ and
\[
u \in C^0 ( ( 0, T ); H^2_{ \mathrm{loc} } ( \Omega ) ) \cap C^1 ( ( 0, T ); H^1_{ \mathrm{loc} } ( \Omega ) ) \text{,}
\]
the following inequality holds almost everywhere on $( 0, T ) \times B^\Omega_{ \varepsilon_0 } ( x_0 )$,
\begin{align}
\label{eq.pointC} e^{ -2 \lambda f } | ( \pm \partial_t + \Delta_\sigma ) u |^2 &\geq 2 ( \partial_t J^t + \nabla \cdot J ) + C \lambda \theta e^{ -2 \lambda f } y^{ 2 p } \, | \nabla u |^2 \\
\notag &\qquad + C e^{ -2 \lambda f } ( \lambda^3 \theta^3 y^{ -1 + 6 p } + \lambda \theta y^{ -2 + 2 p } ) \, u^2 \text{,}
\end{align}
where $J^t$ is a scalar function satisfying
\begin{equation}
\label{eq.Jt} | J^t | \leq \bar{C} e^{ -2 \lambda f } \, | \nabla u |^2 + \bar{C} e^{ -2 \lambda f } \, \lambda^2 \theta^2 y^{-2} \, u^2 \text{,}
\end{equation}
and where $J$ is a vector field satisfying
\begin{align}
\label{eq.J} \nabla y \cdot J &\leq \partial_t ( e^{ -\lambda f } u ) D_y ( e^{ -\lambda f } u ) + \bar{C} e^{ -2 \lambda f } \lambda \theta y^{ 2 p } \, ( D_y u )^2 \\
\notag &\qquad + \bar{C} e^{ -2 \lambda f } \lambda^3 \theta^3 y^{ -2 + 2 p } \, u^2 \text{.}
\end{align}
\end{theorem}

\begin{proof}
To simply the upcoming presentation, we will use $C', C''$ to denote positive constants---depending $T$, $\Omega$, $\sigma$, $p$---whose values can change between lines.
Moreover, we will only treat the backward heat operator $\partial_t + \Delta_\sigma$, as the proof for the forward heat operator $-\partial_t + \Delta_\sigma$ is entirely analogous.

We begin by defining the quantities
\begin{equation}
\label{eql.pointC_0} P_\sigma := \partial_t + \Delta_\sigma \text{,} \qquad v := e^{ -\lambda f } u \text{,} \qquad z > 0 \text{,}
\end{equation}
where the precise value of $z$ is to be determined.
However, by taking $z$ to be large enough and $\varepsilon_0$ to be a small enough, we then have on $B^\Omega_{ \varepsilon_0 } ( x_0 )$:
\begin{equation}
\label{eql.pointC_1} y = d_\Gamma \ll 1 \text{,} \qquad | \nabla^2 y | \ll z \text{,} \qquad | \nabla^2 ( | w |^2 ) | \ll z \text{.}
\end{equation}
In addition, we note the following identities (note here we used \eqref{eq.ortho}):
\begin{align}
\label{eql.pointC_2} \nabla f &= \theta \, [ y^{ 2 p } \nabla y + \nabla ( |w|^2 ) ] \text{,} \\
\notag D_y f &= \theta \, y^{ 2 p } \text{,} \\
\notag \nabla^2 f &= \theta \, [ 2 p y^{ -1 + 2 p } \, ( \nabla y \otimes \nabla y ) + y^{ 2 p } \nabla^2 y + \nabla^2 ( | w |^2 ) ] \text{,} \\
\notag \Delta f &= \theta \, [ 2 p y^{ -1 + 2 p } + y^{ 2 p } \Delta y + \Delta ( | w |^2 ) ] \text{.}
\end{align}

First, using \eqref{eql.pointC_0} and \eqref{eql.pointC_2}, we expand $P_\sigma u$ as
\begin{align}
\label{eql.pointC_3} e^{ -\lambda f } P_\sigma u &= Sv + \Delta v + \mc{A}_0 \, v + \mc{E}_0 \, v \text{,} \\
\notag Sv &= \partial_t v + 2 \lambda \, \nabla f \cdot \nabla v + 2 \lambda \theta ( p y^{ -1 + 2 p } - z y^{ 2 p } ) \, v \text{,} \\
\notag \mc{A}_0 &= \lambda \partial_t f + \lambda^2 | \nabla f |^2 + \sigma y^{-2} \text{,} \\
\notag \mc{E}_0 &= \lambda \theta [ ( 2 z + \Delta y ) y^{ 2 p } + \Delta ( | w |^2 ) ] \text{.}
\end{align}
Multiplying the first part of \eqref{eql.pointC_3} by $Sv$, and noting that
\begin{align*}
e^{ -\lambda f } P_\sigma u \, Sv &\leq \tfrac{1}{2} e^{ -2 \lambda f } | P_\sigma u |^2 + \tfrac{1}{2} | Sv |^2 \text{,} \\
\mc{E}_0 \, v Sv &\leq \tfrac{1}{2} | Sv |^2 + C' \lambda^2 \theta^2 y^{ 4 p } \, v^2 \text{,}
\end{align*}
where we also recalled \eqref{eql.pointC_1}, we then obtain the bound
\begin{align}
\label{eql.pointC_4} \tfrac{1}{2} e^{ -2 \lambda f } | P_\sigma u |^2 &\geq \tfrac{1}{2} | S v |^2 + \Delta v S v + \mc{A}_0 \, v Sv + \mc{E}_0 \, v Sv \\
\notag &\geq \Delta v S v + \mc{A}_0 \, v Sv - C' \lambda^2 \theta^2 y^{ 4 p } \, v^2 \text{.}
\end{align}

By some direct computations, along with \eqref{eql.pointC_2}, the first two terms on the right-hand side of \eqref{eql.pointC_4} can be further expanded as
\begin{align*}
\Delta v S v &= \Delta v \partial_t v + 2 \lambda \Delta v ( \nabla f \cdot \nabla v ) + 2 \lambda \theta ( p y^{ -1 + 2 p } - z y^{ 2 p } ) \, v \Delta v \\
&= - \partial_t \big( \tfrac{1}{2} | \nabla^2 v | \big) + \nabla \cdot ( \nabla v \partial_t v ) + \nabla \cdot [ 2 \lambda \, \nabla v ( \nabla f \cdot \nabla v ) - \lambda \, \nabla f \, | \nabla v |^2 ] \\
&\qquad + \nabla \cdot \big[ 2 \lambda \theta ( p y^{ -1 + 2 p } - z y^{ 2 p } ) \, v \nabla v - \lambda \theta \nabla ( p y^{ -1 + 2 p } - z y^{ 2 p } ) \, v^2 \big] \\
&\qquad - 2 \lambda \, ( \nabla v \cdot \nabla^2 f \cdot \nabla v ) + \lambda \theta \, [ ( 2 z + \Delta y ) y^{ 2 p } + \Delta ( | w |^2 ) ] \, | \nabla v |^2 \\
&\qquad + \lambda \theta \Delta ( p y^{ -1 + 2 p } - z y^{ 2 p } ) \, v^2 \text{,} \\
\mc{A}_0 \, v S v &= \tfrac{1}{2} \mc{A}_0 \, \partial_t ( v^2 ) + \lambda \mc{A}_0 \, \nabla f \cdot \nabla ( v^2 ) + 2 \lambda \theta ( p y^{ -1 + 2 p } - z y^{ 2 p } ) \mc{A}_0 \, v^2 \\
\notag &= \partial_t \big( \tfrac{1}{2} \mc{A}_0 \, v^2 \big) + \nabla \cdot ( \lambda \mc{A}_0 \nabla f \, v^2 ) - \tfrac{1}{2} \partial_t \mc{A}_0 \, v^2 - \lambda \nabla f \cdot \nabla \mc{A}_0 \, v^2 \\
\notag &\qquad - \lambda \theta [ ( 2 z + \Delta y ) y^{ 2 p } + \Delta ( | w |^2 ) ] \mc{A}_0 \, v^2 \text{.}
\end{align*}
Combining \eqref{eql.pointC_1}, \eqref{eql.pointC_4}, and the above, we then obtain
\begin{align}
\label{eql.pointC_10} \tfrac{1}{2} e^{-2 \lambda f} | P_\sigma u |^2 &\geq \partial_t J_t + \nabla \cdot J^0 - 2 \lambda \, ( \nabla v \cdot \nabla^2 f \cdot \nabla v ) \\
\notag &\qquad + 2 ( 1 - \delta ) z \lambda \theta y^{ 2 p } \, | \nabla v |^2 + \mc{A} \, v^2 - C' \lambda^2 \theta^2 y^{ 4 p } \, v^2 \text{,} \\
\notag \mc{A} &= - \tfrac{1}{2} \partial_t \mc{A}_0 - \lambda \nabla f \cdot \nabla \mc{A}_0 + \lambda \theta \Delta ( p y^{ -1 + 2 p } - z y^{ 2 p } ) \\
\notag &\qquad - \lambda \theta [ ( 2 z + \Delta y ) y^{ 2 p } + \Delta ( | w |^2 ) ] \mc{A}_0 \text{,} \\
\notag J^t &= - \tfrac{1}{2} \, | \nabla v |^2 + \tfrac{1}{2} \mc{A}_0 \, v^2 \text{,} \\
\notag J_0 &= \nabla v \partial_t v + 2 \lambda \, \nabla v ( \nabla f \cdot \nabla v ) - \lambda \, \nabla f \, | \nabla v |^2 \\
\notag &\qquad + 2 \lambda \theta ( p y^{ -1 + 2 p } - z y^{ 2 p } ) \, v \nabla v + \lambda \mc{A}_0 \nabla f \, v^2 \\
\notag &\qquad - \lambda \theta \nabla ( p y^{ -1 + 2 p } - z y^{ 2 p } ) \, v^2 \text{,}
\end{align}
where the precise value of $\delta > 0$ will be determined later, but $\delta$ can be chosen to be arbitrarily small by making $z$ sufficiently large.

We now expand the first-order terms in the inequality of \eqref{eql.pointC_10}.
Applying the identities \eqref{eql.pointC_2} and then recalling \eqref{eql.pointC_1}, we conclude that
\begin{align}
\label{eql.pointC_12} &2 ( 1 - \delta ) z \lambda \theta y^{ 2 p } \, | \nabla v |^2 - 2 \lambda \, ( \nabla v \cdot \nabla^2 f \cdot \nabla v ) \\
\notag &\quad \geq - 4 p \lambda \theta y^{ -1 + 2 p } \, ( D_y v )^2 + \lambda \theta y^{ 2 p } \, \big[ \nabla v \cdot \big( \delta z I - 2 \nabla^2 y \big) \cdot \nabla v \big] \\
\notag &\quad\qquad - 2 \lambda \theta \, \big[ \nabla v \cdot \nabla^2 ( |w|^2 ) \cdot \nabla v \big] + 2 ( 1 - \delta ) z \lambda \theta y^{ 2 p } \, ( D_y v )^2 \\
\notag &\quad \geq C'' \lambda \theta y^{ 2 p } \, | \nabla v |^2 - 4 p \lambda \theta y^{ -1 + 2 p } \, ( D_y v )^2 + z \lambda \theta y^{ 2 p } \, ( D_y v )^2 \text{.}
\end{align}
We apply Proposition \ref{hardy_gen} to the last two terms on the right-hand side of \eqref{eql.pointC_12}, with $q := -\frac{1}{2} + p$ and $q := p$.
Recalling also \eqref{eq.ortho} and \eqref{eq.hessian}, we then obtain
\begin{align}
\label{eql.pointC_13} &- 4 p \lambda \theta y^{ -1 + 2 p } \, ( D_y v )^2 + 2 ( 1 - \delta ) z \lambda \theta y^{ 2 p } \, ( D_y v )^2 \\
\notag &\quad \geq \nabla \cdot [ - 4 p ( 1 - p ) \lambda \theta y^{ -2 + 2 p } \nabla y \, v^2 + \tfrac{1}{2} ( 1 - \delta ) ( 1 - 2 p ) z \lambda \theta y^{ -1 + 2 p } \nabla y \, v^2 ] \\
\notag &\quad\qquad - 4 p ( 1 - p )^2 \lambda \theta y^{ -3 + 2 p } \, v^2 + \tfrac{1}{2} ( 1 - \delta ) ( 1 - 2 p )^2 z \lambda \theta y^{ -2 + 2 p } \, v^2 \\
\notag &\quad\qquad + 4 p ( 1 - p ) \lambda \theta y^{ -2 + 2 p } \Delta y \, v^2 - C' \lambda \theta y^{ -1 + 2 p } \, v^2 \text{.}
\end{align}
(Notice we used that $p < 0$ and $z > 0$.)
Combining \eqref{eql.pointC_1} and \eqref{eql.pointC_10}--\eqref{eql.pointC_13} yields
\begin{align}
\label{eql.pointC_20} \tfrac{1}{2} e^{ -2 \lambda f } | P_\sigma u |^2 &\geq ( \partial_t J^t + \nabla \cdot J ) + C'' \lambda \theta y^{ 2 p } \, | \nabla v |^2 + \mc{A} \, v^2 \\
\notag &\qquad - 4 p ( 1 - p )^2 \lambda \theta y^{ -3 + 2 p } \, v^2 + 4 p ( 1 - p ) \lambda \theta y^{ -2 + 2 p } \Delta y \, v^2 \\
\notag &\qquad + \tfrac{1}{2} ( 1 - \delta ) ( 1 - 2 p )^2 z \lambda \theta y^{ -2 + 2 p } \, v^2 - C' \lambda^2 \theta^2 y^{ -1 + 2 p } \, v^2 \text{,} \\
\notag J &= \nabla v \partial_t v + 2 \lambda \, \nabla v ( \nabla f \cdot \nabla v ) - \lambda \, \nabla f \, | \nabla v |^2 + \lambda \mc{A}_0 \nabla f \, v^2 \\
\notag &\qquad + 2 \lambda \theta ( p y^{ -1 + 2 p } - z y^{ 2 p } ) \, v \nabla v - \lambda \theta \nabla ( p y^{ -1 + 2 p } - z y^{ 2 p } ) \, v^2 \\
\notag &\qquad - 4 p ( 1 - p ) \lambda \theta y^{ -2 + 2 p } \nabla y \, v^2 \\
\notag &\qquad + \tfrac{1}{2} ( 1 - \delta ) ( 1 - 2 p ) z \lambda \theta y^{ -1 + 2 p } \nabla y \, v^2 \text{,}
\end{align}
provided $z$ is chosen to be sufficiently large.

We now turn our attention to $\mc{A}$.
First, by \eqref{eq.ortho}, \eqref{eql.pointC_2}, and \eqref{eql.pointC_3},
\begin{align}
\label{eql.pointC_21} \mc{A}_0 &= \mc{A}_0^\star + \lambda^2 \theta^2 \, | \nabla ( | w |^2 ) |^2 + \lambda \theta' \, | w |^2 \text{,} \\
\notag \mc{A}_0^\star &= \sigma y^{-2} + \lambda^2 \theta^2 y^{ 4 p } + \tfrac{1}{ 1 + 2 p } \lambda \theta' y^{ 1 + 2 p } \text{.}
\end{align}
(In particular, $\mc{A}_0^\star$ represents the part of $\mc{A}_0$ that is arising from $y$.)
Using that $w$ denotes bounded coordinates, we then conclude
\begin{equation}
\label{eql.pointC_22} \mc{A}_0 \geq \sigma y^{-2} + \lambda^2 \theta^2 y^{ 4 p } - C' \lambda^2 \theta^2 \text{.}
\end{equation}
Note that in the last step, we used the inequality (which follows from \eqref{eq.f})
\begin{equation}
\label{eql.pointC_23} | \theta^{(k)} | \lesssim_k \theta^{k+1} \text{.}
\end{equation}
Moreover, \eqref{eql.pointC_21} and \eqref{eql.pointC_23} also yield, for the first term of $\mc{A}$,
\begin{equation}
\label{eql.pointC_24} - \tfrac{1}{2} \partial_t \mc{A}_0 \geq - C' \lambda^2 \theta^3 y^{ 4 p } \text{.}
\end{equation}

Next, we apply direct computations using \eqref{eql.pointC_2} and \eqref{eql.pointC_23} to obtain
\begin{align}
\label{eql.pointC_26} - \lambda \nabla f \cdot \nabla \mc{A}_0 &= - \lambda \theta y^{ 2 p } \nabla y \cdot \nabla \mc{A}_0^\star - \lambda \theta \, \nabla ( | w |^2 ) \cdot \nabla \mc{A}_0^\star \\
\notag &\qquad - \lambda \nabla f \cdot \nabla \big[ \lambda^2 \theta^2 \, | \nabla ( | w |^2 ) |^2 + \lambda \theta' \, | w |^2 \big] \\
\notag &\geq 2 \sigma \lambda \theta y^{ -3 + 2 p } - 4 p \lambda^3 \theta^3 y^{ -1 + 6 p } - C' \lambda^3 \theta^3 \, y^{ 4 p } \text{.}
\end{align}
In particular, the two main terms on the right-hand side of \eqref{eql.pointC_26} come from the term $- \lambda \theta y^{ 2 p } \nabla y \cdot \nabla \mc{A}_0^\star$, while the term $- \lambda \theta \, \nabla ( | w |^2 ) \cdot \nabla \mc{A}_0^\star$ vanishes due to \eqref{eql.pointC_1}.
(Crucially, without \eqref{eql.pointC_1}, this latter term leads to quantities that are $\smash{ O ( y^{-3} ) }$ and too singular to treat.)
Furthermore, as $w$ is bounded, all the terms arising from $w$ are strictly less singular and hence can be treated as error terms.

Similar computations can be done for the remaining terms of $\mc{A}$ in \eqref{eql.pointC_10}:
\begin{align}
\label{eql.pointC_27} &\lambda \theta \Delta ( p y^{ -1 + 2 p } - z y^{ 2 p } ) - \lambda \theta [ ( 2 z + \Delta y ) y^{ 2 p } + \Delta ( | w |^2 ) ] \mc{A}_0 \\
\notag &\quad \geq 2 p ( 1 - 2 p ) ( 1 - p ) \lambda \theta y^{ -3 + 2 p } + p ( 1 - 2 p ) \lambda \theta ( 2 z - \Delta y ) \, y^{ -2 + 2 p } \\
\notag &\quad \qquad - \sigma \lambda \theta [ ( 2 z + \Delta y ) \, y^{ -2 + 2 p } + \Delta ( | w |^2 ) \, y^{-2} ] \\
\notag &\quad \qquad - C' \lambda \theta \, y^{ -1 + 2 p } - C' \lambda^3 \theta^3 y^{ 6 p } \text{.}
\end{align}
Combining \eqref{eql.pointC_10} and \eqref{eql.pointC_24}--\eqref{eql.pointC_27}, we then obtain
\begin{align}
\label{eql.pointC_28} \mc{A} &\geq [ 2 \sigma + 2 p ( 1 - 2 p ) ( 1 - p ) ] \lambda \theta y^{ -3 + 2 p } + [ - 2 \sigma + 2 p ( 1 - 2 p ) ] z \lambda \theta y^{ -2 + 2 p } \\
\notag &\qquad + [ - \sigma - p ( 1 - 2 p ) ] \Delta y \, \lambda \theta y^{ -2 + 2 p } - \sigma \Delta ( | w |^2 ) \, \lambda \theta y^{-2} \\
\notag &\qquad - 4 p \lambda^3 \theta^3 y^{ -1 + 6 p } - C' \lambda \theta y^{ -1 + 2 p } - C' \lambda^3 \theta^3 y^{ 6 p } \text{.}
\end{align}
From \eqref{eql.pointC_20} and \eqref{eql.pointC_28}, we then have
\begin{align}
\label{eql.pointC_30} \tfrac{1}{2} e^{-2 \lambda f} | P_\sigma u |^2 &\geq ( \partial_t J_t + \nabla \cdot J ) + C'' \lambda \theta y^{ 2 p } \, | \nabla v |^2 - 4 p \lambda^3 \theta^3 y^{ -1 + 6 p } \, v^2 \\
\notag &\qquad - C' \lambda^3 \theta^3 y^{ 6 p } \, v^2 + 2 [ \sigma - p ( 1 - p ) ] \lambda \theta y^{ -3 + 2 p } \, v^2 \\
\notag &\qquad + \big[ {- 2 \sigma} + \tfrac{1}{2} ( 1 - \delta ) + 2 \delta p - 2 ( 1 + \delta ) p^2 \big] z \lambda \theta y^{ -2 + 2 p } \, v^2 \\
\notag &\qquad + ( -\sigma + 3 p - 6 p^2 ) \lambda \theta \Delta d_\Gamma \, d_\Gamma^{ -2 + 2 \kappa } \, v^2 \\
\notag &\qquad - \sigma \Delta ( | w |^2 ) \, \lambda \theta y^{-2} \, v^2 - C' \lambda^3 \theta^3 y^{ -1 + 2 p } \, v^2 \text{.}
\end{align}

To treat the above, we first note that \eqref{eq.ass_p} implies
\begin{equation}
\label{eql.pointC_31} \sigma \geq p ( 1 - p ) \text{,} \qquad 2 [ \sigma - p ( 1 - p ) ] \lambda \theta y^{ -3 + 2 p } \, v^2 \geq 0 \text{.}
\end{equation}
Next, we claim that the following inequality holds:
\begin{equation}
\label{eql.pointC_32} {- 2 \sigma} + \tfrac{1}{2} ( 1 - \delta ) + 2 \delta p - 2 ( 1 + \delta ) p^2 > 0
\end{equation}
The proof of \eqref{eql.pointC_32} splits into two cases.
First, if $\sigma < 0$, then a direct computation combined with \eqref{eq.ass_p} shows that taking any $\delta \leq \frac{1}{2}$ and $p \leq \kappa$ results in \eqref{eql.pointC_32}.
On the other hand, when $\sigma > 0$, then the assumption $\sigma < \frac{1}{4}$ implies that we can take $\delta$ small enough (and hence, $z$ sufficiently large) and $p$ close enough to $0$ (recall \eqref{eq.ass_p}) such that \eqref{eql.pointC_32} again holds, completing the proof of \eqref{eql.pointC_32}.

In particular, \eqref{eql.pointC_32} implies we can choose $z$ large enough so that
\begin{align}
\label{eql.pointC_33} C'' y^{ -2 + 2 p } &\leq \big[ {- 2 \sigma} + \tfrac{1}{2} ( 1 - \delta ) + 2 \delta p - 2 ( 1 + \delta ) p^2 \big] z y^{ -2 + 2 p } \\
\notag &\qquad + ( -\sigma + 3 p - 6 p^2 ) \Delta y \, y^{ -2 + 2 p } - \sigma \Delta ( | w |^2 ) \, y^{-2} \text{.}
\end{align}
Moreover, as long as $\varepsilon_0$ is small enough, then \eqref{eql.pointC_1} also implies
\begin{equation}
\label{eql.pointC_34} -4 p y^{ -1 + 6 p } - C' y^{ 6 p } - C' y^{ -1 + 2 p } \geq C'' y^{ -1 + 6 p } \text{.}
\end{equation}
Combining \eqref{eql.pointC_30}, \eqref{eql.pointC_31}, \eqref{eql.pointC_33}, and \eqref{eql.pointC_34} then results in the following inequality:
\begin{align}
\label{eql.pointC_40} \tfrac{1}{2} e^{ -2 \lambda f } | P_\sigma u |^2 &\geq ( \partial_t J_t + \nabla \cdot J ) + C'' \lambda \theta y^{ 2 p } \, | \nabla v |^2 \\
\notag &\qquad + C'' \lambda \theta y^{ -2 + 2 p } \, v^2 + C'' \lambda^3 \theta^3 y^{ -1 + 6 p } \, v^2 \text{.}
\end{align}
Rewriting \eqref{eql.pointC_40} in terms of $u$ using \eqref{eql.pointC_0}, and noting the bound
\[
e^{ -2 \lambda f } y^{ 2 p } \, | \nabla u |^2 \leq C' y^{ 2 p } \, | \nabla v |^2 + C' \lambda^2 \theta^2 y^{ 6 p } v^2 \text{,}
\]
which is a consequence of \eqref{eql.pointC_2} and the boundedness of the coordinates $w$, we then obtain the desired inequality \eqref{eq.pointC} once $\lambda$ is made sufficiently large.

It remains to show the inequalities \eqref{eq.Jt} and \eqref{eq.J}.
First, for \eqref{eq.Jt}, we expand the formula for $J_t$ in \eqref{eql.pointC_10} in order to estimate
\begin{align*}
| J^t | &\leq \tfrac{1}{2} | \nabla v |^2 + \tfrac{1}{2} \mc{A}_0 \, v^2 \\
&\leq C' e^{ -2 \lambda f } \, | \nabla u |^2 + C' e^{ -2 \lambda f } \, \lambda^2 \theta^2 y^{-2} \, u^2 \text{,}
\end{align*}
where we recalled \eqref{eql.pointC_0}, \eqref{eql.pointC_2}, \eqref{eql.pointC_21}, and \eqref{eql.pointC_23}.
For \eqref{eq.J}, we expand \eqref{eql.pointC_20}:
\begin{align*}
\nabla y \cdot J &= \partial_t v D_y v + 2 \lambda D_y v \, ( \nabla f \cdot \nabla v ) - \lambda D_y f \, | \nabla v |^2 + \lambda \mc{A}_0 D_y f \, v^2 \\
&\qquad + 2 \lambda \theta ( \kappa y^{ -1 + 2 p } - z y^{ 2 p } ) \, v D_y v - \lambda \theta \, \nabla y \cdot \nabla ( p y^{ -1 + 2 p } - z y^{ 2 p} ) \, v^2 \\
&\qquad - 4 p ( 1 - p ) \lambda \theta y^{ -2 + 2 p } \, v^2 + \tfrac{1}{2} ( 1 - \delta ) ( 1 - 2 p ) z \lambda \theta y^{ -1 + 2 p } \, v^2 \text{.}
\end{align*}
The right-hand side of the above can be expanded using \eqref{eql.pointC_2} and bounded:
\begin{align}
\label{eql.pointC_43} \nabla y \cdot J &\leq \partial_t v D_y v + 2 \lambda \theta y^{ 2 \kappa } \, ( D_y v )^2 + 2 \lambda D_y v \, \nabla ( | w |^2 ) \cdot \nabla v - \lambda \theta y^{ 2 p } \, | \nabla v |^2 \\
\notag &\qquad + C' \lambda \theta y^{ -1 + 2 p } \, |v| | D_y v | + C' \lambda \theta y^{ -2 + 2 p } \, v^2 \\
\notag &\leq | \partial_t v D_y v | + C' \lambda \theta y^{ 2 p } \, ( D_y v )^2 + C' \lambda^3 \theta^3 y^{ -2 + 2 p } \, v^2 \text{.}
\end{align}
(In particular, the term in the right-hand side of \eqref{eql.pointC_43} containing $\nabla ( |w|^2 ) \cdot \nabla v$ was absorbed into the negative $| \nabla v |^2$-term.)
Thus, recalling also \eqref{eql.pointC_0}, we have
\begin{align*}
\notag \nabla y \cdot J &\leq \partial_t ( e^{ -\lambda f } u ) D_y ( e^{ -\lambda f } u ) + C' e^{ -2 \lambda f } \, \lambda \theta y^{ 2 p } \, ( D_y u )^2 \\
&\qquad + C' e^{ -2 \lambda f } \, \lambda^3 \theta^3 y^{ -2 + 2 p } \, u^2 \text{,}
\end{align*}
which is precisely the inequality \eqref{eq.J}.
\end{proof}

\begin{remark}
Note that the ability to handle non-convex $\Gamma$ in Lemma \ref{thm.pointC}, in contrast to \cite{new}, arises from the fact that $z$ in \eqref{eql.pointC_0} can be chosen to be arbitrarily large.
In \cite[Lemma 2.7]{new}, the admissible range of $z$ is also constrained from above.
\end{remark}

\begin{remark}
The key to reducing the regularity of $\Gamma$ to $C^2$, in contrast to \cite{new}, is that we modified the zero-order part of $S v$ in \eqref{eql.pointC_3} to contain only the most singular terms, which do not contain any derivatives of $y$.
The remaining zero-order contributions, which do involve derivatives of $y$, were left in the ``error" coefficient $\mc{E}_0$, which could be absorbed into leading terms without taking more derivatives.
In effect, this allows us to only take two derivatives of $y$ in the proof of Lemma \ref{thm.pointC}.
\end{remark}

\begin{remark}
Note that the proof of Lemma \ref{thm.pointC} breaks down in the limit $\sigma \nearrow \frac{1}{4}$.
In particular, the crucial inequality \eqref{eql.pointC_32} requires that $\sigma < \frac{1}{4}$.
\end{remark}

\subsection{The integrated estimate}

Under sufficient regularity, we can then integrate our pointwise estimate \eqref{thm.pointC} to obtain our integrated local Carleman estimate:

\begin{theorem}[Local Carleman estimate] \label{T.Carleman}
Let $p \in ( -\frac{1}{2}, 0 )$ satisfy \eqref{eq.ass_p}, fix a point $x_0 \in \Gamma$, and let $f := f_p, \theta$ be as in Definition \ref{def.f}.
Then, there exist $C, \bar{C}, \lambda_0, \varepsilon_0 > 0$ (depending on $T$, $\Omega$, $\sigma$, $p$) such that the following Carleman estimate holds,
\begin{align}
\label{eq.Carleman} &C \int_{ (0, T) \times [ B^\Omega_\varepsilon ( x_0 ) \cap \{ y > \delta \} ] } e^{ -2 \lambda f } \big[ \lambda \theta y^{ 2 p } | \nabla u |^2 + ( \lambda^3 \theta^3 y^{ -1 + 6 p } + \lambda \theta y^{ -2 + 2 p } ) \, u^2 \big] \\
\notag &\quad \leq \bar{C} \int_{ ( 0, T ) \times [ B^\Omega_\varepsilon ( x_0 ) \cap \{ y = \delta \} ] } e^{ -2 \lambda f } [ \lambda \theta y^{ 2 p } \, ( D_y u )^2 + \lambda^3 \theta^3 y^{ -2 + 2 p } \, u^2 ] \\
\notag &\quad\qquad + \int_{ ( 0, T ) \times [ B^\Omega_\varepsilon ( x_0 ) \cap \{ y = \delta \} ] } \partial_t ( e^{ -\lambda f } u ) D_y ( e^{ -\lambda f } u ) \\
\notag &\quad\qquad + \int_{ ( 0, T ) \times [ B^\Omega_\varepsilon ( x_0 ) \cap \{ y > \delta \} ] } e^{ -2 \lambda f } | ( \pm \partial_t + \Delta_\sigma ) u |^2 \text{,}
\end{align}
for all $\lambda \geq \lambda_0$ and $0 < \delta \ll \varepsilon \leq \varepsilon_0$, and for all functions
\begin{equation}
\label{eq.Carleman_reg} u \in C^0 ( [ 0, T ]; \mf{D} ( A_\sigma^2 ) ) \cap C^1 ( ( 0, T ); \mf{D} ( A_\sigma ) )
\end{equation}
such that $u$ vanishes in a neighborhood of $\Omega \cap \partial B_\varepsilon ( x_0 )$.
\end{theorem}

\begin{proof}
Let $C, \bar{C}, \lambda_0, \varepsilon_0$ as in Theorem \ref{thm.pointC}, and fix any $\lambda \geq \lambda_0$ and $0 < \delta \ll \varepsilon \leq \varepsilon_0$.
To make our notations more concise, we also set
\begin{equation}
\label{eq.om_delta} \omega_{ >\delta } := B^\Omega_\varepsilon ( x_0 ) \cap \{ y > \delta \} \text{,} \qquad \omega_{ = \delta } := B^\Omega_\varepsilon ( x_0 ) \cap \{ y = \delta \} \text{.}
\end{equation}

For any $u$ satisfying \eqref{eq.Carleman_reg}, we integrate \eqref{eq.pointC} over $( 0, T ) \times \omega_\delta$, with $0 < \delta \ll 1$, and we then apply the divergence theorem to obtain
\begin{align}
\label{eql.carleman_0} &C \int_{ ( 0, T ) \times \omega_{ > \delta } } e^{ -2 \lambda f } \big[ \lambda \theta y^{ 2 p } \, | \nabla u |^2 + ( \lambda^3 \theta^3 y^{ -1 + 6 p } + \lambda \theta y^{ -2 + 2 p } ) \, u^2 \big] \\
\notag &\quad \leq \int_{ ( 0, T ) \times \omega_{ > \delta } } e^{ -2 \lambda f } | ( \pm \partial_t + \Delta_\sigma ) u |^2 - 2 \int_{ ( 0, T ) \times \partial \omega_{ > \delta } } \nu \cdot J \\
\notag &\quad \qquad - 2 \int_{ \{ T \} \times \omega_{ > \delta } } J^t + 2 \int_{ \{ 0 \} \times \omega_{ > \delta } } J^t \text{,}
\end{align}
where $\nu$ denotes the outer unit normal of $\omega_{ > \delta }$.

For the last term on the right-hand side of \eqref{eql.carleman_0}, we recall \eqref{eq.Jt} and bound
\begin{align}
\label{eql.carleman_1} \int_{ \{ 0 \} \times \omega_{ >\delta } } | J^t | &\leq \int_{ \{ 0 \} \times \omega_{ >\delta } } e^{ -2 \lambda f } ( | \nabla u |^2 + \lambda^2 \theta^2 d_\Gamma^{-2} u ) \\
\notag &= 0 \text{,}
\end{align}
where the last step is due to \eqref{hardy_main}, the $H^1$-boundedness of $u$, and the exponential vanishing of $e^{ -2 \lambda f }$ at $t = 0$.
An analogous estimate also yields
\begin{equation}
\label{eql.carleman_2} \int_{ \{ T \} \times \omega_{ > \delta } } | J^t | = 0 \text{.}
\end{equation}
Lastly, since $u$ is assumed to vanish near $\Omega \cap \partial B_\varepsilon ( x_0 )$, we then have
\begin{align}
\label{eql.carleman_3} - \int_{ ( 0, T ) \times \partial \omega_{ > \delta } } \nu \cdot J &= \int_{ ( 0, T ) \times \omega_{ = \delta } } \nabla y \cdot J \\
\notag &\leq \bar{C} \int_{ (0, T) \times \omega_{ = \delta } } e^{ -2 \lambda f } [ \lambda \theta y^{ 2 p } \, ( D_y u )^2 + \lambda^3 \theta^3 y^{ -2 + 2 p } \, u^2 ] \\
\notag &\quad\qquad + \int_{ (0, T) \times \omega_{ = \delta } } \partial_t ( e^{ -\lambda f } u ) D_y ( e^{ -\lambda f } u ) \text{,}
\end{align}
where we applied \eqref{eq.J}.
The desired \eqref{eq.Carleman} now follows from \eqref{eql.carleman_0}--\eqref{eql.carleman_3}.
\end{proof}

\begin{remark}
Note all the boundary integrals over $\{ y = \delta \}$ in \eqref{eq.Carleman} are well-defined by the usual trace theorems \cite{Evans}, since $y$ is a positive constant on this hypersurface.
\end{remark}

\section{Unique continuation} \label{S.uc}

Next, we apply the Carleman estimate of Theorem \ref{T.Carleman} to derive a unique continuation property for homogeneous, singular backward heat equations from Problem (O).
While such a property is of independent interest, it also functions as a crucial step toward our main approximate controllability result.

Our precise unique continuation result is as follows:

\begin{theorem}[Unique continuation] \label{T.UC}
Fix $x_0 \in \Gamma$ and a sufficiently small $\varepsilon > 0$ (depending on $T$, $\Omega$, $\sigma$).
Furthermore, let $u_T \in H^1_0 ( \Omega )$, and let $u$ denote the corresponding solution to Problem (O).
If the Neumann trace $\mc{N}_\sigma u$ vanishes everywhere on $( 0, T ) \times B^\Gamma_\varepsilon ( x_0 )$, then $u$ in fact vanishes everywhere on $( 0, T ) \times \Omega$.
\end{theorem}

\subsection{Proof of Theorem \ref{T.UC}}

First, suppose $u_T \in C^\infty_0 ( \Omega )$, so that the solution $u$ satisfies \eqref{eq.Carleman_reg}.
Let $f$ be as in Definition \ref{def.f} (with our given $x_0$ and $\varepsilon$).
In addition, fix $0 < f_0 \ll 1$ and a smooth cutoff function
\begin{equation}
\label{cutoff} \chi: (0, f_0 ) \to [0, 1] \text{,} \qquad \chi (s) := \begin{cases} 1 & s \in \big( 0, \tfrac{f_0}{2} \big) \text{,} \\ 0 & s \in \big( \tfrac{3 f_0}{4}, f_0 \big) \text{,} \end{cases}
\end{equation}
with $\chi'$ denoting differentiation of $\chi$ in its parameter $s$.
We then have
\begin{align}
\label{boundPsigma} | ( \partial_t + \Delta_\sigma ) ( u \cdot \chi(f) ) | &\lesssim | ( \partial_t + \Delta_\sigma ) u | \chi(f) + ( | \chi'(f) | + | \chi''(f) | ) \big( | \nabla u | + |u| \big) \\
\notag &\lesssim | \nabla u | + y^{-1} \, |u| \text{,}
\end{align}
where we recalled the equation \eqref{heat_ex}, that the coefficients $( X, V )$ lie in $\mc{Z}$, and that $\chi'(f)$ and $\chi''(f)$ are supported in the region $\{\smash{ \frac{f_0}{2} \leq f \leq \frac{ 3 f_0 }{4} }\}$.

Now, as $f$ vanishes at $x_0$ and is positive nearby, then by taking $f_0$ to be sufficiently small, we can ensure that $u \cdot \chi (f)$ vanishes on $\Omega \cap \partial B_\varepsilon ( x_0 )$.
Thus, we can apply the Carleman estimate of Theorem \ref{T.Carleman} to $u \cdot \chi(f)$ in order to obtain
\begin{align*}
& \int_{ \{ f \leq \frac{f_0}{2} \} \cap \{ y \geq \delta \} } e^{-2 \lambda f} \big( \lambda \theta y^{ 2 p } | \nabla u |^2 + \lambda \theta y^{ -2 + 2 p } \, u^2 \big) \\
&\quad \leq \bar{C} \int_{ ( 0, T ) \times [ B^\Omega_\varepsilon ( x_0 ) \cap \{ y = \delta \} ] } e^{ -2 \lambda f } [ \lambda \theta y^{ 2 p } \, ( D_y u )^2 + \lambda^3 \theta^3 y^{ -2 + 2 p } \, u^2 ] \\
&\quad\qquad + \int_{ ( 0, T ) \times [ B^\Omega_\varepsilon ( x_0 ) \cap \{ y = \delta \} ] } \partial_t ( e^{ -\lambda f } u ) D_y ( e^{ -\lambda f } u ) \\
&\quad\qquad + C \int_{ \{ 0 \leq f \leq \frac{ 3 f_0 }{4} \} \cap \{ y \geq \delta \} } e^{-2 \lambda f} \big( | \nabla u |^2 + y^{-2} u^2 \big)
 \text{,}
\end{align*}
where $\varepsilon$ and $\lambda$ are sufficiently small and large, respectively; where $0 < \delta \ll \varepsilon$ and $C, \bar{C} > 0$; and where $p$ is chosen to satisfy both \eqref{eq.ass_p} and $2 p - \kappa > -\frac{1}{2}$.
Further expanding the integrand $\partial_t ( e^{ -\lambda f } u ) D_y ( e^{ -\lambda f } u )$ in the above, we then have
\begin{align}
\label{intC_loc_intermediate} & \int_{ \{ f \leq \frac{f_0}{2} \} \cap \{ y \geq \delta \} } e^{-2 \lambda f} \big( \lambda \theta y^{ 2 p } | \nabla u |^2 + \lambda \theta y^{ -2 + 2 p } \, u^2 \big) \\
\notag &\quad \leq \bar{C} \int_{ ( 0, T ) \times [ B^\Omega_\varepsilon ( x_0 ) \cap \{ y = \delta \} ] } e^{ -2 \lambda f } [ \lambda \theta y^{ 2 p } \, ( D_y u )^2 + \lambda^3 \theta^3 y^{ -2 + 2 p } \, u^2 ] \\
\notag &\quad\qquad + \int_{ ( 0, T ) \times [ B^\Omega_\varepsilon ( x_0 ) \cap \{ y = \delta \} ] } \partial_t ( y^{ -\kappa } u ) [ w_1 \, y^{ 2 \kappa } D_y ( y^{ -\kappa } u ) + w_0 \, y^{ -1 + \kappa } u ] \\
\notag &\quad\qquad + C \int_{ \{ 0 \leq f \leq \frac{ 3 f_0 }{4} \} \cap \{ y \geq \delta \} } e^{-2 \lambda f} \big( | \nabla u |^2 + y^{-2} u^2 \big)
 \text{,}
\end{align}
where the weights $w_0$ and $w_1$ both satisfy \eqref{stupid_weight}.
(Note that the terms arising from $\partial_t$ hitting $e^{ -\lambda f }$ can be absorbed into the first term on the right-hand side.)

At this point, we have only established \eqref{intC_loc_intermediate} for $u_T \in C^\infty_0 ( \Omega )$.
However, since $\delta \leq y \lesssim 1$ on $\{ y \geq \delta \}$, and since $y$ is constant on $\{ y = \delta \}$, then each term in \eqref{intC_loc_intermediate} is in fact controlled by the $H^1$-norm of $u_T$.
(This is a consequence of Propositions \ref{wp_strict} and \ref{neumann_trace}--\ref{stupid_boundary}.)
Thus, by an approximation, we conclude that \eqref{intC_loc_intermediate} still holds when $u_T \in H^1_0 ( \Omega )$.
From here on, we will assume the general setting $u_T \in H^1_0 ( \Omega )$.

We next take the limit $\delta \searrow 0$ in \eqref{intC_loc_intermediate}.
Since $\mc{N}_\sigma u \equiv 0$ on $( 0, T ) \times B^\Gamma_\varepsilon ( x_0 )$ by assumption, then Proposition \ref{neumann_limits} and \eqref{eq.ass_p} imply
\begin{align}
\label{intC_loc_11} &\lim_{ \delta \searrow 0 } \int_{ ( 0, T ) \times [ B^\Omega_\varepsilon ( x_0 ) \cap \{ y = \delta \} ] } e^{ -2 \lambda f } [ \lambda \theta y^{ 2 p } \, ( D_y u )^2 + \lambda^3 \theta^3 y^{ -2 + 2 p } \, u^2 ] \\
\notag &\quad \lesssim \lambda^3 \lim_{ \delta \searrow 0} \delta^{ 2 p - 2 \kappa } \int_{ ( 0, T ) \times [ B^\Omega_\varepsilon ( x_0 ) \cap \{ y = \delta \} ] } | \eta^\delta ( y^\kappa D_y u ) |^2 + | \eta^\delta ( y^{ -1 + \kappa } u ) |^2 \\
\notag &\quad \lesssim \lambda^3 \lim_{ \delta \searrow 0 } \delta^{ 2 p + 1 } \| u_T \|_{ H^1 ( \Omega ) }^2 \\
\notag &\quad = 0 \text{.}
\end{align}
Likewise, since $2 p - \kappa > -\frac{1}{2}$ and $w_0, w_1$ satisfy \eqref{stupid_weight}, then \eqref{stupid_limit} implies
\begin{equation}
\label{intC_loc_12} \lim_{ \delta \searrow 0 } \int_{ ( 0, T ) \times [ B^\Omega_\varepsilon ( x_0 ) \cap \{ y = \delta \} ] } \partial_t ( y^{ -\kappa } u ) [ w_1 \, y^{ 2 \kappa } D_y ( y^{ -\kappa } u ) + w_0 \, y^{ -1 + \kappa } u ] = 0 \text{.}
\end{equation}
Combining \eqref{intC_loc_11}--\eqref{intC_loc_12}, then \eqref{intC_loc_intermediate} in the limit becomes
\[
\lambda \int_{ \{ f \leq \frac{f_0}{2} \} } e^{-2 \lambda f} \theta \big( | \nabla u |^2 + y^{ -2 } \, u^2 \big) \lesssim \int_{ \{ 0 \leq f \leq \frac{ 3 f_0 }{4} \} } e^{-2 \lambda f} \big( | \nabla u |^2 + y^{-2} u^2 \big) \text{,}
\]

By taking $\lambda$ large enough in the above, then part of the right-hand side can be absorbed into the left-hand side, and we hence obtain
\begin{align}
\label{intC_loc_intermediate1} &\lambda \int_{ \{ f \leq \frac{f_0}{2} \} } e^{-2 \lambda f} \theta \big( | \nabla u |^2 + y^{-2} \, u^2 \big) \lesssim \int_{ \{ \frac{f_0}{2} \leq f \leq \frac{ 3 f_0 }{4} \} } e^{-2 \lambda f} \big( | \nabla u |^2 + y^{-2} \, u^2 \big) \text{.}
\end{align}
Observe now that since
\[
e^{-2 \lambda f} \begin{cases}
\geq e^{- \lambda f_0 } & 0 \leq f \leq \frac{ f_0 }{2} \text{,} \\
\leq e^{- \lambda f_0 } & \frac{ f_0 }{2} \leq f \leq \frac{ 3 f_0 }{4} \text{,}
\end{cases}
\]
then the above combined with \eqref{intC_loc_intermediate1} yields
\begin{align}
\label{intC_loc_intermediate2} &\lambda e^{-\lambda f_0} \int_{ \{ f \leq \frac{f_0}{2} \} } \theta \big( | \nabla u |^2 + y^{-2} \, u^2 \big) \lesssim e^{-\lambda f_0} \int_{ \{ \frac{f_0}{2} \leq f \leq \frac{ 3 f_0 }{4} \} } \big( | \nabla u |^2 + y^{-2} u^2 \, \big) \text{.}
\end{align}
The exponential factors in \eqref{intC_loc_intermediate2} now cancel, so that taking $\lambda \nearrow \infty$ now yields
\begin{equation}
\label{intC_loc_intermediate3} u |_{ \{ f \leq \frac{ f_0 }{2} \} } \equiv 0 \text{.}
\end{equation}

In particular, by \eqref{eq.f}, the above implies that for $0 < \alpha \ll T$, we have that $u$ vanishes on $( \alpha, T - \alpha ) \times \omega^\alpha$, for some neighborhood $\omega^\alpha$ of $x_0$ in $\Omega$.
Since $u$ vanishes in a neighborhood away from the boundary, then standard parabolic unique continuation results yield that $u \equiv 0$ on $( \alpha, T - \alpha ) \times \Omega$.
(Indeed, the key point is that the equation \eqref{heat_ex} is non-singular away from the boundary, so that the above immediately follows from Carleman estimates for non-singular parabolic equations, e.g.\ \cite{LRLB, Yamamoto}, or alternatively via the Carleman estimate of \cite{BZuazua} for interior control.)
Finally, letting $\alpha \searrow 0$ in the above completes the proof.


\section{Approximate Control} \label{S.proof}

In this section, we prove our main approximate controllability result, Theorem \ref{T.approx0}.
For convenience, we first restate Theorem \ref{T.approx0} in the language of Section \ref{S.wp}:

\begin{theorem} \label{thm.approx0}
Suppose Assumptions \ref{ass.domain} and \ref{ass.sigma} hold.
Also, let $( Y, W ) \in \mc{Z}_0$, and fix any open $\omega \subseteq \Gamma$.
Then, given any $T > 0$, any $v_0, v_T \in H^{-1} (\Omega)$, and any $\epsilon > 0$, there exists $v_d \in L^2 ( (0, T) \times \Gamma )$, supported within $( 0, T ) \times \omega$, such that the solution $v$ of Problem (C), with the above $v_0$ and $v_d$, satisfies
\begin{equation}
\label{eq.approx_ctrl} \| v(T) - v_T \|_{ H^{-1} ( \Omega ) } \leq \epsilon \text{.}
\end{equation}
\end{theorem}

\subsection{Proof of Theorem \ref{thm.approx0}}

First, recall that it suffices to only consider the case $v_0 \equiv 0$.
(To see this, we suppose $v_0 \not\equiv 0$, and we let $v_\ast$ be the value at $t = T$ of the solution to Problem (C), with this $v_0$ and with $v_d \equiv 0$.
Then, by linearity, a control $v_d \in L^2 ( ( 0, T ) \times \Gamma )$ taking $v_0$ to be $\varepsilon$-close to $v_T$ at time $t = T$ is equivalent to $v_d$ taking zero initial data to be $\varepsilon$-close to $v_T - v_\ast$ at time $t = T$.)

We also state the following characterization of approximate controls:

\begin{proposition} \label{thm.ctrl_char}
A solution $v$ of Problem (C), with $v_0 \equiv 0$ and $v_d \in L^2 ( ( 0, T ) \times \Gamma )$, satisfies $v (T) = \psi_T \in H^{-1} ( \Omega )$ if and only if for any $u_T \in H^1_0 ( \Omega )$, we have
\begin{equation}
\label{eq.ctrl_char} \int_{ ( 0, T ) \times \Gamma } \mc{N}_\sigma u \, v_d + \int_\Omega u_T \psi_T = 0 \text{,}
\end{equation}
where $u$ is the (semigroup) solution to Problem (O), with data $u_T$ as above, and with the lower-order coefficients $( X, V )$ defined as in \eqref{dual_XV}.
\end{proposition}

\begin{proof}
This is an immediate consequence of \eqref{weak_id}, along with the fact that $v_T$ is completely characterized by its action on every $u_T \in H^1_0 ( \Omega )$.
\end{proof}

Next, we define the functional $I_\epsilon: H^1_0 ( \Omega ) \rightarrow \R$, with $\epsilon>0$, by
\begin{equation}
\label{eq.functional} I_\epsilon ( u_T ) := \epsilon \| u_T \|_{ H^1 ( \Omega ) } + \tfrac{1}{2} \int_{ ( 0, T ) \times \omega } ( \mc{N}_\sigma u )^2 + \int_\Omega u_T v_T \text{,}
\end{equation}
with $u$ again being the solution to Problem (O) with $u_T$ as above and $( X, V )$ as in \eqref{dual_XV}.
Then, the minimizers of $I_\epsilon$ yield the desired approximate controls:

\begin{lemma} \label{thm.variational}
Suppose $\varphi_T$ is a minimizer of $I_\epsilon$, and let $\varphi$ be the solution of Problem (O), with $u_T := \varphi_T$, and with $( X, V )$ as in \eqref{dual_XV}.
Then, the solution $v$ to Problem (C), with initial data $v_0 \equiv 0$ and Dirichlet trace
\begin{equation}
\label{eq.variational} v_d := \begin{cases}
  \mc{N}_\sigma \varphi & \text{on $( 0, T ) \times \omega$,} \\
  0 & \text{on $( 0, T ) \times ( \Gamma \setminus \omega )$,}
\end{cases}
\end{equation}
satisfies the estimate \eqref{eq.approx_ctrl}.
\end{lemma}

\begin{proof}
Consider linear variations of $\varphi_T$---for any $s \in \R \setminus \{ 0 \}$ and $u_T \in H^1_0 ( \Omega )$,
\begin{align}
\label{eql.variational_0} \frac{ I_\epsilon ( \varphi_T + s u_T ) - I_\epsilon ( \varphi_T ) }{s} &= \frac{ \epsilon [ \| \varphi_T + s u_T \|_{ H^1 ( \Omega ) } - \| \varphi_T \|_{ H^1 ( \Omega ) } ] }{s} \\
\notag &\qquad + \int_{ ( 0, T ) \times \omega } \big[ \mc{N}_\sigma \varphi \mc{N}_\sigma u + \tfrac{1}{2} s ( \mc{N}_\sigma u )^2 \big] + \int_\Omega u_T v_T \text{,}
\end{align}
where $u$ denotes the solution to Problem (O), with final data $u_T$, and with $( X, V )$ as in \eqref{dual_XV}.
Since $\varphi_T$ is a minimizer of $I_\epsilon$, we have
\begin{equation}
\label{eql.variational_1} I_\epsilon ( \varphi_T ) \leq I_\epsilon ( \varphi_T + s u_T ) \text{.}
\end{equation}

As a result, if $s > 0$, then \eqref{eql.variational_0}--\eqref{eql.variational_1} imply
\begin{align}
\label{eql.variational_10} 0 &\leq \limsup_{ s \rightarrow 0+ } \frac{ I_\epsilon ( \varphi_T + s u_T ) - I_\epsilon ( \varphi_T ) }{s} \\
\notag &\leq \epsilon \| u_T \|_{ H^1 ( \Omega ) } + \int_{ ( 0, T ) \times \omega } \mc{N}_\sigma \varphi \mc{N}_\sigma u + \int_\Omega u_T v_T \text{.}
\end{align}
Similarly, if $s < 0$, then an analogous derivation using \eqref{eql.variational_0}--\eqref{eql.variational_1} yields
\begin{align}
\label{eql.variational_11} 0 &\geq \liminf_{ s \rightarrow 0- } \frac{ I_\epsilon ( \varphi_T + s u_T ) - I_\epsilon ( \varphi_T ) }{s} \\
\notag &\geq - \epsilon \| u_T \|_{ H^1 ( \Omega ) } + \int_{ ( 0, T ) \times \omega } \mc{N}_\sigma \varphi \mc{N}_\sigma u + \int_\Omega u_T v_T \text{.}
\end{align}
Thus, combining \eqref{eql.variational_10}--\eqref{eql.variational_11}, we obtain
\begin{equation}
\label{eql.variational_20} \left| \int_{ ( 0, T ) \times \omega } \mc{N}_\sigma \varphi \mc{N}_\sigma u + \int_\Omega u_T v_T \right| \leq \epsilon \| u_T \|_{ H^1 ( \Omega ) } \text{.}
\end{equation}

Finally, letting $v$ be as in the lemma statement, we then have
\begin{align*}
\int_{ ( 0, T ) \times \omega } \mc{N}_\sigma \varphi \mc{N}_\sigma u &= \int_{ ( 0, T ) \times \Gamma } v_d \mc{N}_\sigma u \\
&= - \int_\Omega u_T v (T)
\end{align*}
by Proposition \ref{thm.ctrl_char}, and hence \eqref{eql.variational_20} becomes
\[
\left| \int_\Omega u_T [ v (T) - v_T ] \right| \leq \epsilon \| u_T \|_{ H^1 ( \Omega ) } \text{.}
\]
The desired \eqref{eq.app_ctrl} now follows by varying over all $u_T \in H^1_0 ( \Omega )$.
\end{proof}

In light of Lemma \ref{thm.variational}, it remains only to show that $I_\epsilon$ indeed has a minimizer.
This is accomplished in the subsequent lemma, for which the proof relies crucially on the unique continuation property of Theorem \ref{T.UC}:

\begin{lemma} \label{thm.minimizer}
$I_\epsilon$ has a minimizer $\varphi_T$.
\end{lemma}

\begin{proof}
Propositions \ref{wp_strict} and \ref{neumann_trace} imply that $I_\epsilon$ is continuous and convex.
Thus, it suffices to show $I_\epsilon$ is also coercive, that is, given any sequence $( u_{ T, k } )_{ k \geq 0 }$ in $H^1_0 ( \Omega )$,
\begin{equation}
\label{eql.minimizer_0} \lim_{ k \rightarrow \infty } \| u_{ T, k } \|_{ H^1 ( \Omega ) } = +\infty \quad \Rightarrow \quad \lim_{ k \rightarrow \infty } I_\epsilon ( u_{ T, k } ) = +\infty \text{.}
\end{equation}

Assume now the left-hand side of \eqref{eql.minimizer_0}, and consider the normalized sequence
\begin{equation}
\label{eql.minimizer_1} \varphi_{ T, k } := \frac{ u_{ T, k } }{ \| u_{ T, k } \|_{ H^1 ( \Omega ) } } \text{,} \qquad k \gg 1 \text{.}
\end{equation}
Let $\varphi_k$ be the solution to Problem (O), with $u_T := \varphi_{ T, k }$, and with $( X, V )$ defined as in \eqref{dual_XV}.
Then, the definition \eqref{eq.functional} yields
\begin{equation}
\label{eql.minimizer_2} \frac{ I_\epsilon ( u_{ T, k } ) }{ \| u_{ T, k } \|_{ H^1 ( \Omega ) } } = \epsilon + \| u_{ T, k } \|_{ H^1 ( \Omega ) } \int_{ ( 0, T ) \times \omega } ( \mc{N}_\sigma \varphi_k )^2 + \int_\Omega \varphi_{ T, k } v_T \text{.}
\end{equation}
From here, the proof splits into two cases.

First, suppose that
\begin{equation}
\label{eql.minimizer_10} \liminf_{ k \rightarrow \infty } \int_{ ( 0, T ) \times \omega } ( \mc{N}_\sigma \varphi_k )^2 > 0 \text{.}
\end{equation}
Then, \eqref{eql.minimizer_2} and \eqref{eql.minimizer_10} together imply
\[
\lim_{ k \rightarrow \infty } \frac{ I_\epsilon ( u_{ T, k } ) }{ \| u_{ T, k } \|_{ H^1 ( \Omega ) } } = + \infty \text{,}
\]
and the desired \eqref{eql.minimizer_0} immediately follows.

Next, for the remaining case, suppose
\begin{equation}
\label{eql.minimizer_20} \liminf_{ k \rightarrow \infty } \int_{ ( 0, T ) \times \omega } ( \mc{N}_\sigma \varphi_k )^2 = 0 \text{.}
\end{equation}
Since the sequence $( \varphi_{ T, k } )_{ k \gg 1 }$ is uniformly bounded in $H^1_0 ( \Omega )$, there is a subsequence that converges weakly to some $\varphi_T \in H^1_0 ( \Omega )$; let $\varphi$ then be the solution to Problem (O), with $u_T := \varphi_T$.
Testing $\varphi_k - \varphi$ with solutions of Problem (C) with initial data $v_0 \equiv 0$, the identity \eqref{weak_id} can be applied to conclude that $\mc{N}_\sigma \varphi_k$ converges weakly to $\mc{N}_\sigma \varphi$ in $L^2 ( ( 0, T ) \times \Gamma )$.
Therefore, recalling that continuity implies weak lower semicontinuity, the above then implies
\begin{align}
\label{eql.minimizer_21} \int_{ ( 0, T ) \times \omega } ( \mc{N}_\sigma \varphi )^2 &\leq \liminf_{ k \rightarrow \infty } \int_{ ( 0, T ) \times \omega } ( \mc{N}_\sigma \varphi_k )^2 \\
\notag &= 0 \text{.} 
\end{align}
Applying the unique continuation property of Theorem \ref{T.UC} to \eqref{eql.minimizer_21}, with an arbitrary $x_0 \in \omega$ and $\varepsilon > 0$ small enough so that $B_\varepsilon^\Gamma ( x_0 ) \subseteq \omega$, then yields $\varphi \equiv 0$, and hence $\varphi_T \equiv 0$.
Combining this with \eqref{eql.minimizer_2}, we conclude that
\begin{align*}
\liminf_{ k \rightarrow \infty } \frac{ I_\epsilon ( u_{ T, k } ) }{ \| u_{ T, k } \|_{ H^1 ( \Omega ) } } &\geq \liminf_{ k \rightarrow \infty } \left[ \epsilon + \int_\Omega \varphi_{ T, k } v_T \right] \\
&= \epsilon \text{.}
\end{align*}
The above, in particular, implies the right-hand side of \eqref{eql.minimizer_0}.
\end{proof}

The conclusion of Theorem \ref{thm.approx0} is now immediate, since Lemma \ref{thm.minimizer} produces a minimizer for $I_\epsilon$, which by Lemma \ref{thm.variational} yields a control $v_d$ satisfying \eqref{eq.approx_ctrl}.

\raggedbottom

\section*{Acknowledgments}

The authors thank Alberto Enciso for several prolific discussions that led to the writing of this manuscript.

\end{document}